\documentclass[12pt]{amsart}
\usepackage{amsfonts}
\usepackage{amsmath}
\usepackage{amssymb,latexsym}
\usepackage[mathcal]{eucal}
\usepackage{amscd}
\usepackage[pdftex,bookmarks,colorlinks,breaklinks]{hyperref}

\input xy
\xyoption{all}

\newcommand{\Bcal}{\mathcal{B}}
\newcommand{\Ccal}{\mathcal{C}}

\newcommand{\Kcal}{\mathcal{K}}

\newcommand{\Pcal}{\mathcal{P}}
\newcommand{\Qcal}{\mathcal{Q}}
\newcommand{\Rcal}{\mathcal{R}}

\newcommand{\Ucal}{\mathcal{U}}
\newcommand{\Vcal}{\mathcal{V}}

\newcommand{\Xcal}{\mathcal{X}}
\newcommand{\Ycal}{\mathcal{Y}}
\newcommand{\Zcal}{\mathcal{Z}}

\newcommand{\Z}{\mathbb{Z}}

\newcommand{\R}{\mathbb{R}}

\newcommand{\N}{\mathbb{N}}
\newcommand{\T}{\mathbb{T}}
\newcommand{\E}{\mathbb{E}}

\newcommand{\Eb}{\mathbf{E}}

\newcommand{\Kb}{\mathbf{K}}

\newcommand{\Wb}{\mathbf{W}}
\newcommand{\Xb}{\mathbf{X}}
\newcommand{\Yb}{\mathbf{Y}}
\newcommand{\Zb}{\mathbf{Z}}

\newcommand{\al}{\alpha}
\newcommand{\Ga}{\Gamma}
\newcommand{\ga}{\gamma}
\newcommand{\del}{\delta}

\newcommand{\ep}{\epsilon}
\newcommand{\sig}{\sigma}

\newcommand{\la}{\lambda}

\newcommand{\om}{\omega}
\newcommand{\Om}{\Omega}

\newcommand{\ol}{\overline}

\newcommand{\br}{\vspace{3 mm}}

\newcommand{\tri}{\bigtriangleup}
\newcommand{\rest}{\upharpoonright}

\newcommand{\Aut}{{\rm{Aut\,}}}

\newcommand{\id}{{\rm{id}}}

\newcommand{\diam}{{\rm{diam\,}}}

\newcommand{\Sym}{{\rm{Sym\,}}}

\newcommand{\Prob}{{\rm{Prob}}}

\swapnumbers
\theoremstyle{plain}
\newtheorem{thm}{Theorem}[section]

\newtheorem{lem}[thm]{Lemma}
\newtheorem{prop}[thm]{Proposition}

\newtheorem{claim}{Claim}

\theoremstyle{definition}
\newtheorem{defn}[thm]{Definition}

\newtheorem{rmk}[thm]{Remark}
\newtheorem{rmks}[thm]{Remarks}

\newtheorem{setup}[thm]{Set up}

\newtheorem{exa}[thm]{Example}

\begin{document}

%%%%%%%%%%%%%%%%%%%%%%%%%%%%%%%%%%%%%%%%%%%%%%%%%%%%%%%%%%%%%%%%%%%%
%%%%%%%%%%%%%%%%%%%   T H E    T I T L E    %%%%%%%%%%%%%%%%%%%%%%%%%%
%%%%%%%%%%%%%%%%%%%%%%%%%%%%%%%%%%%%%%%%%%%%%%%%%%%%%%%%%%%%%%%%%%%%

\title[A generic distal tower of arbitrary countable height]
%{A distal tower of arbitrary countable height}
{A generic distal tower of arbitrary countable height over an arbitrary infinite ergodic system}

\author{Eli Glasner and Benjamin Weiss}

\address{Department of Mathematics\\
     Tel Aviv University\\
         Tel Aviv\\
         Israel}
\email{glasner@math.tau.ac.il}
\address {Institute of Mathematics\\
 Hebrew University of Jerusalem\\
Jerusalem\\
 Israel}
\email{weiss@math.huji.ac.il}

%
%\begin{date}
%{January 30, 2020}
%\end{date}

%%%%%%%%%%%%%%%%%%%%%%%%%%%%%%%%%%%%%%%%%%%%%%%%%%%%%%%%%%%%%%%%%%%%%
%%%%%%%%%%%%%%%%%%%%%%%%  CONTENTS   %%%%%%%%%%%%%%%%%%%%%%%%%%%%%%%%
%%%%%%%%%%%%%%%%%%%%%%%%%%%%%%%%%%%%%%%%%%%%%%%%%%%%%%%%%%%%%%%%%%%%%

%\tableofcontents
\setcounter{secnumdepth}{2}

%\addtocontents{toc}{subsection}{\protect\hspace{0.5cm}}

%%%%%%%%%%%%%%%%%%%%%%%%%%%%%%%%%%%%%%%%%%%%%%%%%%%%%%%%%%%%%%%%%%%%%
%%%%%%%%%%%%%%%%%%%%%%%%%%%%  TEXT   %%%%%%%%%%%%%%%%%%%%%%%%%%%%%%%%
%%%%%%%%%%%%%%%%%%%%%%%%%%%%%%%%%%%%%%%%%%%%%%%%%%%%%%%%%%%%%%%%%%%%%

\setcounter{section}{0}
%\setcounter{page}{0}

%%%%%%%%%%%%%%%%%%%%%%%%%%%%%%%%%%%%%%%%%%%%%%%%%%%%%%%%%%%%%%%%%%%%%
%%%%%%%%%%%%%%%%%%%       Inroduction         %%%%%%%%%%%%%%%%%%%%%%%
%%%%%%%%%%%%%%%%%%%%%%%%%%%%%%%%%%%%%%%%%%%%%%%%%%%%%%%%%%%%%%%%%%%%%

\begin{abstract}
We show the existence, over an arbitrary infinite ergodic $\Z$-dynamical system,
of a generic ergodic relatively distal extension of arbitrary countable rank
and arbitrary infinite compact extending groups (or more generally, infinite quotients of compact groups)
in its canonical distal tower.
\end{abstract}
 


\begin{date}
{September 14, 2020}
\end{date}

\maketitle

%%%%%%%%%%%%%%%%%%%%%%%%%%%%%%%%%%%%%%%%%%%%%%%%%%%%%%%%%%%%%%%%%%%%%
%%%%%%%%%%%%%%%%%%%%%%%%  CONTENTS   %%%%%%%%%%%%%%%%%%%%%%%%%%%%%%%%
%%%%%%%%%%%%%%%%%%%%%%%%%%%%%%%%%%%%%%%%%%%%%%%%%%%%%%%%%%%%%%%%%%%%%

\tableofcontents
\setcounter{secnumdepth}{2}

%\addtocontents{toc}{subsection}{\protect\hspace{0.5cm}}

\setcounter{section}{0}
%\setcounter{page}{0}

%%%%%%%%%%%%%%%%%%%%%%%%%%%%%%%%%%%%%%%%%%%%%%%%%%%%%%%%%%%%%%%%%%%%%
%%%%%%%%%%%%%%%%%%%       Inroduction         %%%%%%%%%%%%%%%%%%%%%%%
%%%%%%%%%%%%%%%%%%%%%%%%%%%%%%%%%%%%%%%%%%%%%%%%%%%%%%%%%%%%%%%%%%%%%

\section*{Introduction}

It would be hard to exaggerate the importance and impact of Harry Furstenberg's 1963 paper
``The structure of distal flows", \cite{F-63}.
In this revolutionary work Furstenberg started what we call today the ``structure theory"
of dynamical systems. 
We recall that a topologically distal dynamical system $(X,T)$,
with $X$ a metric compact space and $T$ a self homeomorphism,
is called {\em distal} if the only proximal pairs in $X$ are the diagonal pairs; \ i.e.
if $\lim T^{n_i} x=\lim T^{n_i} x'$ for a pair $x,x'\in X$ and
a sequence $n_i\in \Z$, then $x=x'$. Furstenberg's
distal structure theorem asserts that every minimal distal
system $(X,T)$ has, uniquely, a structure of an inverse limit of
a family of factors $\{(X_\al,T):\al<\eta\}$ directed by a
countable ordinal $\eta$
such that for every $\al<\eta$ the extension $X_{\al+1}\to X_\al$
is a maximal topologically isometric extension.

Whereas in \cite{F-63} the subject of study is that of ``minimal flows",
so in the domain of topological dynamics, the works \cite{Z-76a}, \cite{Z-76b} and \cite{F-77} and \cite{F-81}
introduce and prove an analogous theorem in the context of ergodic theory,
called today the Furstenberg-Zimmer structure theorem for ergodic systems, 
%(see also  \cite{Z-76a}, \cite{Z-76b}),
which is the main tool for Furstenberg's ergodic version of Szemer\'edi's theorem, \cite{Sz}.

Roughly speaking, an extension $\Xb \to \Yb$ of ergodic dynamical systems is {\em compact} 
when $\Xb$ is a skew product over $\Yb$ with fibers all of which have the form of a homogeneous space
$K/H$, where $K$ is a compact group and $H < K$ a closed subgroup.
In the special case when $H$ is trivial the extension is called a {\em group extension}
and in this case $K$ consists of a compact group of automorphisms of the system $\Xb$ with
$\Yb \cong \Xb/K$. In fact, it turns out that every compact extension is of this form.

In his works \cite{F-77} and \cite{F-81} Furstenberg defines an ergodic 
measure theoretical system $\Xb$ to be distal
if it is obtained as an iteration of countably many compact extensions, where in the
(possibly transfinite construction) at a limit ordinal one takes an inverse limit.

Now W. Parry in his 1967 paper \cite{Pa} suggested an intrinsic
definition of measure distality. He defines
a property of measure dynamical systems, called
``admitting a separating sieve", which imitates the intrinsic
definition of topological distality as follows: 

\begin{quote}
%\begin{defn}\label{defn-ssieve}
Let $\Xb = (X, \mathcal{X}, \mu, T)$ be an ergodic system.
A sequence $A_1\supset A_2\supset \cdots$ of sets in $\Xcal$
with $\mu(A_n)>0$ and $\mu(A_n)\to 0$,
is called a {\em separating sieve \/}\label{def-sep-sieve-ext}
if there exists a
subset $X_0\subset X$ with $\mu(X_0)=1$ such that for every
$x,x'\in X_0$, the condition ``for every $n\in \N$ there exists $k \in \Z$ with
$T^k x, T^k x'\in A_n$" implies $x=x'$.
%\end{defn}
\end{quote}

In 1976 in two fundamental papers \cite{Z-76a}, \cite{Z-76b}
R. Zimmer developed the theory of distal systems
and distal extensions for a general locally compact acting
group. He showed that, as in the topologically distal case,
systems admitting Parry's separating sieve are exactly
those with generalized discrete spectrum, that is those systems
which are exhausted by their Furstenberg tower of
compact extensions.

An extension of dynamical systems $\pi: \Xb \to \Yb$ is called a {\em {relatively weakly mixing extension}}
when the corresponding relative product
$(X \underset Y \times X, \mu \underset \nu \times \mu, T)$ is ergodic.
%We also say in this case that $\Xb$ is {\em relatively weakly mixing over $\Yb$}.
In particular $\Xb$ is {\em {weakly mixing}} when the product system $\Xb \times \Xb$
is ergodic. 

\br 

The Furstenberg-Zimmer structure theorem says that every ergodic  dynamical system
has a unique structure as a relative weakly mixing extension of a distal system,
and that the latter admits a uniquely defined canonical distal tower.
%Roughly speaking, an extension $\Xb \to \Yb$ of ergodic dynamical systems is {\em compact} 
%when $\Xb$ is a skew product over $\Yb$ with fibers all of which have the form of a homogeneous space
%$K/H$, where $K$ is a compact group and $H < K$ a closed subgroup.
%In the special case when $H$ is trivial the extension is called a {\em group extension}
%and in this case $K$ consists of a compact group of automorphisms of the system $\Xb$ with
%$\Yb \cong \Xb/K$. In fact, it turns out that every compact extension is of this form.
%
%
%A dynamical system $\mathbf{X}$ is called {\em distal}
%if it is an iteration of countably many compact extensions, where in the
%(possibly transfinite construction) at a limit ordinal one takes an inverse limit.
%\footnote{
%W. Parry in his 1967 paper \cite{Pa} suggested an intrinsic
%definition of measure distality. He defines
%a property of measure dynamical systems, called
%``admitting a separating sieve", which imitates the intrinsic
%definition of topological distality. 
%In 1976 in two outstanding papers \cite{Z-76a}, \cite{Z-76b}
%R. Zimmer developed the theory of distal systems
%and distal extensions for a general locally compact acting
%group. He showed that, as in the topologically distal case,
%systems admitting Parry's separating sieve are exactly
%those with generalized discrete spectrum, that is those systems
%which are exhausted by their Furstenberg tower of
%compact extensions.}.
The so called {\em canonical distal tower} is unique if at each stage one takes the maximal compact extension
(within $\Xb$).
The height of this tower (a countable ordinal) is called the {\em rank} of the distal system $\Xb$.

\br

In \cite{BF} Beleznay and Foreman show that for every countable ordinal 
$\eta$ there is an ergodic distal system of rank $\eta$. 
Our main result in the present work is as follows.

\begin{thm}\label{thm-main}
Given an arbitrary countable ordinal $\eta$ and a transfinite sequence of pairs
$$
\{(K_\al, H_\al) : \al = 0, \ {\text{and}} \  \al < \eta,\ \al \  {\text{ a successor ordinal}} \},
$$
where for each $\al$ $K_\al$ is an infinite compact second countable topological group
and $\{e\} \leq H_\al < K_\al$, a proper closed subgroup of infinite index in $K_\al$, 
with the only requirement that $K_0$ be an infinite monothetic compact group
(and $H_0 = \{e\}$),
there exists (generically) an ergodic distal system $\Xb$ of rank $\eta$ such that, in its canonical distal tower,
for each successor ordinal $\al$ the extension $\Xb_\al \to \Xb_{\al -1}$ is a $K_{\al-1}/ H_{\al-1}$-extension
(here $\Xb_0$ is the trivial one point system).
\end{thm}

Note that in \cite{BF} the authors prove their result via an explicit
inductive construction on the infinite torus $\T^\N$, so that at each stage
the extending compact group is $\T$, whereas in
our construction the compact groups 
(and more generally also their quotients), which serve as building
blocks for the tower, are arbitrary. Moreover, we show that at each
successor ordinal the construction yields a generic extension.

We also note that Theorem \ref{thm-main} is in fact the best result one can prove 
regarding the Furstenberg-Zimmer structure theorem, since the requirement
that all the groups $K_\al$ be infinite is really necessary, see Remark \ref{finite} below.

\br

The relative version of the Furstenberg-Zimmer theorem 
says that for any given extension of ergodic systems $\Yb \to \Zb$,
there is a diagram  $\Yb \to \Yb_{rd} \to \Zb$, 
where $\Yb_{rd}$ is the largest relative distal extension of $\Zb$ in $\Yb$,
with a uniquely defined canonical relatively distal tower, 
and such that the extension 
$\Xb \to \Yb_{rd}$ is relatively weakly mixing.
As Theorems \ref{max2} and \ref{max} of the present work 
are proven over an arbitrary infinite ergodic system $\Zb$,
it follows that our proof of Theorem \ref{thm-main} works in the relative case as well,
producing a canonical relative distal tower of height $\eta$ over $\Zb$
(in fact, in view of Theorem \ref{generica}, when $\Zb$ is infinite ergodic, the assumption on
$K_1$ can be relaxed, we only need it to be infinite).

\begin{thm}\label{general-relative}
Let $\Zb$ be an infinite ergodic system.
Given an arbitrary countable ordinal $\eta$ and a transfinite sequence of pairs
$$
\{(K_\al, H_\al) : \al = 0 , \ {\text{and}}\ \al < \eta,\ \al \  {\text{ a successor ordinal}} \},
$$
where for each $\al$ $K_\al$ is an infinite compact second countable topological group
and $\{e\} \leq H_\al < K_\al$, a proper closed subgroup of infinite index in $K_\al$, 
%with the only requirement that $[K_1 : H_1]$ be infinite,
%(and $H_1 = \{e\}$),
there exists (generically) an ergodic system $\Xb$ 
which is relatively distal over $\Zb$ of rank $\eta$ 
such that, in its canonical distal tower,
for each successor $\al$ the extension $\Xb_\al \to \Xb_{\al -1}$ is a $K_{\al-1}/ H_{\al-1}$-extension
(here $\Xb_0$ is the system $\Zb$).
\end{thm}

\br

After a preliminary section, where we introduce the basic definitions (Section \ref{sec-preli}),
we describe in Section \ref{sec-rank} the strategy of the proof of the main theorem.
Then, as a preliminary result, whose proof will indicate for the reader 
an essential trait of the general strategy,
we show in Section \ref{sec-toy},
that given an ergodic system $\Xb$ and a compact topological group $G$,
the generic cocycle $\phi : X \to G$ induces an ergodic skew product extension $\Xb_\phi \to \Xb$.
This is a generalization of a theorem of Jones and Parry \cite{JP} where the authors proved
this result for an abelian $G$.

We then go on with the main proof, by stages, in Section  \ref{sec-ranks}.
In Section \ref{sec-cor} we draw some corollaries of Theorem \ref{thm-main}.
In Section \ref{sec-wm-ext} we prove an analogous statement about generic group extensions 
over a weakly mixing system (and more generally over a relatively weakly mixing extension).
Finally, in the last section (Section  \ref{sec-con}) 
%we state some conjectures about possible generalizations of our theorems.
we present a general framework for our results and prove a master theorem
of which most of our main results are consequences.

\br

\section{Some preliminaries and structure theory}\label{sec-preli}

A {\em  dynamical system} (sometimes also called a {\em $\Z$-action}) is a quadruple 
$\Xb = (X, \mathcal{X}, \mu, T)$,
where $(X,\mathcal{X},\mu)$ is a standard probability space
$T$ is an element of the
Polish group $\Aut(X,\mu)$ of invertible measure preserving transformation of
$(X, \mathcal{X}, \mu)$. 
When $\mathbf{X}$ and $\mathbf{Y} = (Y. \mathcal{Y}, \nu, T)$ are two dynamical systems, we say that
$\mathbf{Y}$ is a {\em factor} of $\mathbf{X}$ ( or that 
$\mathbf{X}$ is an {\em extension} of $\mathbf{Y}$) is there is 
a measurable map $\pi : X \to Y$ such that $\pi_*(\mu)=\nu$ and such that
$\pi(T x) = T\pi(x)$ for $\mu$ almost every $x \in X$.
The map $\pi$ is called a {\em factor map} (or an {\em extension}).

The system $\mathbf{X}$ is {\em ergodic} if every $T$-invariant set $A \in \mathcal{X}$
(i.e. $T A = A \pmod \mu$ is trivial :  $\mu(A)(1 -  \mu(A))=0$).

Let $\Yb$ be a dynamical system and $(V,\Vcal,\rho)$
a standard probability space. Let $S \mapsto S_y$ be 
a measurable map $Y \to {\Aut}(V,\rho)$;   then 
$S$ defines a  {\em cocycle}, i.e. a function $\tilde{S} : \Z \times Y \to {\Aut}(V,\rho)$,
$$
\tilde{S}(n,y)=\begin{cases}
S_{T^{n-1} y}\circ \cdots \circ S_{Ty} \circ S_y & {\text{for}}\  n\ge 1\\
\id & {\text{for}}\  n=0\\
S^{-1}_{T^ny} \circ \cdots \circ S^{-1}_{Ty}  & {\text{for}}\  n< 0.
\end{cases}
$$

We define the {\em skew-product system\/}\label{def-skew-prod}
$\Yb\underset{S}{\times}(V,\rho)$ to be the system
$(Y\times V,\Ycal\otimes\Vcal,\mu\times\rho,T_S)$, where
$T_S(y,v)=(Ty, S_y(v))$. 

\br 

In the special case where $V$ is a compact group and $\rho$ is its normalized Haar measure,
any measurable function $\phi : Y \to V$ defines a skew product by the formula:
\begin{align*}
T_\phi(y,v) & = (Ty, \phi(y) v), \ y \in Y, v \in V, \ {\text{and}}\\
T^n_\phi(y,v) & = (T^ny, \phi_n(y) v), \ n \in \Z.
\end{align*}
Here $\tilde{\phi}(n,y) = \phi_n(y) = \phi(T^{n-1}y) \cdots \phi(Ty) \cdot \phi(y)$
for $n > 0$ and a similar formula for $n < 0$.

\br

We have the following basic theorem:

\begin{thm}[Rokhlin]\label{rsp}
Let $\pi : \Xb\to\Yb$ be a factor map of dynamical systems with
$\Xb$ ergodic, then $\Xb$ is isomorphic to a skew-product
over $\Yb$. Explicitly, there exist a standard probability space
$(V,\Vcal,\rho)$ and  a measurable map $S:  Y \to
{\Aut}(V,\rho)$ with $\Xb \cong \Yb \times_S (V,\rho)
=(Y\times V,\Ycal\otimes\Vcal,\nu\times\rho,T_S)$, where
$T_S(y,u)=(Ty,S_y(v))$, and $\pi(y,v) = y$.
\end{thm}

The map $y \mapsto S_y$ is called the {\em Rokhlin cocycle} of the extension $\pi$.

\br

The topology on $\Aut(X,\mu)$ is induced by a complete metric
$$
D(S,T) = \sum_{n \in \N} 2^{-n}  (\mu(SA_n \tri TA_n) + \mu(S^{-1}A_n \tri T^{-1}A_n)), 
$$
with $\{A_n\}_{n \in \N}$ a dense sequence in the measure algebra
$(\mathcal{X},d_\mu)$, where $d_\mu(A,B) = \mu(A \tri B)$.
Equipped with this topology $\Aut(X, \mu)$ is a Polish topological group and we say that
the dynamical system $\mathbf{X}$ is {\em compact} if the 
set $\{T^n : n \in \Z\}$ is a precompact subgroup of $\Aut(X,\mu)$.

\begin{exa}
Let $K$ be a compact monothetic topological group;
i.e. there is a homomorphism $\phi : \Z \to K$ with a dense image,
and we let $T x = ax, \ x \in X$, where $a = \phi(1)$
(so that the image of $\phi$ is the dense subgroup $\{a^n : n \in \Z\}$). 
With $\mathcal{K}$ the algebra of Borel subsets of $K$ and 
$\la$ is the normalized Haar measure on $K$, the system $\Xb = (K, \mathcal{K}, \la, T)$,
 is an ergodic compact dynamical system.
\end{exa}

It turns out that, in fact, every ergodic compact system $\mathbf{X}$ has this form.

The notion of compactness can now be relativized as follows:

\br

%\begin{defn}
An extension $\pi : \Xb \to \Yb$, with $\Xb$ ergodic is a {\em compact extension} 
if there is a compact second countable topological group $K$,
a closed subgroup $H < K$ and a measurable map (sometimes called a {\em cocycle}) 
$\phi : Y  \to K$ such that
$$
\Xb\cong \Yb \times_\phi (K/H,\rho)
=(Y\times K/H,\Ycal\otimes\Kcal,\nu\times\rho,T_\phi),
$$
where $\rho$ is the Haar measure on $K/H$ and 
$T_\phi(y, kH) = (Ty, \phi(y)kH)$.
The cocycle $\phi$ is {\em minimal} if there is no cocycle $\psi  : \Ga \times Y \to K$ cohomologous 
to $\phi$ with $K_\psi \subsetneq K_\phi$. where $K_\phi$ and $K_\psi$ are the closed subgroups of $K$
generated by the ranges of $\phi$ and $\psi$ respectively.
(The cocycles $\phi$ and $\psi$ are {\em cohomologous}
when there is a measurable map $\kappa : Y \to K$ such that
$\psi(Ty) = \kappa(Ty)^{-1}\phi(y) \kappa(y),\ \nu$-a.e.)
% (by the ergodicity of $\Xb) we have $K_\al = K$).
%\end{defn}

\begin{thm}\label{homog}
Given a compact extension $\pi : \Xb \to \Yb$ with $\Xb$ ergodic,
we can always assume that 
$$
\Xb\cong \Yb \times_\phi (K/H,\rho)
=(Y\times K/H,\Ycal\otimes\Kcal,\nu\times\rho,T_\phi),
$$
where the cocycle $\phi$ is minimal with $K_\phi= K$.
The corresponding group extension $\hat{\pi} : \hat{\Xb} \to \Yb$,
with $\hat{X} = Y \times K$, is ergodic and the diagram
\begin{equation*}
\xymatrix
{
\hat \Xb =\Yb  \times_\phi K
 \ar[d]_{\hat\pi} \ar[dr]^{\sig}  & \\
\Yb &  \Xb=\Yb \times_\phi K/H \ar[l]^-{\pi}
}
\end{equation*}
commutes. 
Here $\hat\Xb = \Yb \times_\phi K$
is the group skew-product defined by the cocycle $\phi$,\
i.e. $\hat\mu=\nu\times \hat\rho$ where $\hat\rho$ is Haar measure on $K$,
and $\hat \mu$ is ergodic. The map $\sig : \hat{\Xb} \to \Xb$ 
is the quotient map $\sig(y,k) = (y, kH)$.
\end{thm}
For the proof and more details see e.g. \cite[Corollary 3.27]{G-03}.

\br

This construction can be iterated and a dynamical system $\mathbf{X}$ is called {\em distal} 
if it is an iteration of countably many compact extensions, where in the
(possibly transfinite construction) at a limit ordinal one takes an inverse limit.
The so called {\em canonical distal tower} is unique if at each stage one takes the maximal compact extension
(within $\Xb$).
The height of this tower (a countable ordinal) is called the {\em rank} of the distal system $\Xb$.

\br

An extension of dynamical systems $\pi: \Xb \to \Yb$ is called a {\em {weakly mixing extension}}
when the corresponding relative product
$(X \underset Y \times X, \mu \underset \nu \times \mu, T)$ is ergodic.
In particular $\Xb$ is {\em {weakly mixing}} when the product system $\Xb \times \Xb$
is ergodic. 

We now can state the following

\begin{thm}[Furstenberg-Zimmer structure theorem]
Every ergodic system $\Xb$ has (uniquely) a largest distal factor $\pi : \Xb \to \Yb$
and the extension $\pi$ is a weakly mixing one.
\end{thm}

\br

In \cite{BF} Beleznay and Foreman show that for every countable ordinal 
$\eta$ there is an ergodic distal system of rank $\eta$.
We refer e.g. to \cite{G-03} for more details on structure theory in ergodic theory.

\br

%\section{Distal $\Z$-systems of arbitrary countable rank}\label{sec-rank}

\section{The strategy of the proof of Theorem \ref{thm-main}}\label{sec-rank}

We begin with a few comments on Theorem \ref{thm-main}.

%
%We can now state our our main result as follows:
%
%
%
%\begin{thm}\label{tower}
%Given an arbitrary countable ordinal $\eta$ and a transfinite sequence of pairs
%$$
%\{(K_\al, H_\al) : 0 < \al \leq \eta,\ \al \  {\text{ a successor ordinal}} \},
%$$
%where for each $\al$ $K_\al$ is an infinite compact second countable topological group
%and $\{e\} \leq H_\al < K_\al$, a proper closed subgroup of infinite index in $K_\al$, 
%with the only requirement that $K_1$ be an infinite monothetic compact group
%(and $H_1 = \{e\}$),
%there exists an ergodic distal system $\Xb$ of rank $\eta$ such that, in its canonical distal tower,
%for each successor $\al$ the extension $\Xb_\al \to \Xb_{\al -1}$ is a $K_\al/ H_\al$-extension
%(here $\Xb_0$ is the trivial one point system).
%\end{thm}

\begin{rmks}\label{monothetic}
\begin{enumerate}
\item
Whereas in \cite{BF} the authors prove their result for the special case where 
$K_\al = \mathbb{T} = \R/\Z$ for all $\al$, in
our construction the compact groups (or their quotients) which serve as building
blocks for the tower are arbitrary. 
\item
Furthermore, our constructions yield
generic extensions at each successor ordinal.
\item
Note however that, at the first stage of the tower, 
for the infinite monothetic $K_0$, the set of topological generators
$K_g = \{k \in K_0 : \ol{\{k^n : n \in \Z\}} = K_0\}$ is a dense $G_\del$ subset of $K_0$
iff $K_0$ does not admit a nontrivial finite quotient group.
\end{enumerate}
\end{rmks}

\begin{proof}[Proof of the latter remark]
Let $\{U_m\}_{m \in \N}$ be a basis for the topology of $K_0$.
For each $m$ set
$$
\Ucal_m = \{k \in K_0 : \exists n \in \N, \ k^n \in U_m\}.
$$
Clearly each $\Ucal_m$ is open, and $K_g = \bigcap_{m \in \N} \Ucal_m$.
Thus $K_g$ is always a $G_\del$ set.
We assume that $K_0$ is monothetic so it has at least one generator, say $k_0$;
so that $K_g$ is nonempty.

Now if for each $0 \not = n \in \Z$,  $k^n_0$ is a topological generator, then $\{k_0^n  : n \in \Z\}$
is a subset of $K_g$ and it follows that $K_g$ is a dense $G_\del$ subset of $K_0$.
Otherwise, there is $t \geq 2$ such that the subgroup $N = \ol{\{k_0^{tn}  : n \in \Z\}}$
is a proper subgroup of finite index $[N, K_0] = t$.
So clearly in this case $K_g$ is not dense.
\end{proof}

\br

\begin{defn}
An extension $\Xb \to \Yb$ of dynamical systems (not necessarily ergodic) is
{\em relatively ergodic}, or that {\em $\Xb$ is relatively ergodic over $\Yb$}, 
if every invariant $L_2(\mu)$ function is $\Ycal$ measurable.
\end{defn}
%
%
%
%In the sequel we will repeatedly use the following lemma which, in turn, 
%is based on the characterization of compact extensions in \cite[Theorem 6.13]{F-81}.
%We give a proof based on \cite[Theorem 7.1]{F-77}.

In the sequel we will repeatedly use the following lemma which is explicitly formulated in 
\cite[Lemma 2.8]{BF}.  The authors of \cite{BF} base their proof  on the characterization of compact 
extensions in \cite[Theorem 6.13]{F-81}.
We give here a brief proof based on \cite[Theorem 7.1]{F-77}.

\begin{lem}\label{BF}
Let $\Xb$ be an ergodic system.
Let $\Zb$ be a factor of $\Xb$ and $\Yb$ be a compact extension of $\Zb$  in $\Xb$
(i.e. $\Xb \to \Yb \to \Zb$). Then $\Yb$ is the maximal compact extension of
$\Zb$ in $\Xb$ iff $\Xb \underset{\Zb} { \times} \Xb$ is relatively ergodic over 
$\Yb \underset{\Zb} { \times} \Yb$.
\end{lem}

\begin{proof}
Let $\tilde{\Yb}$ be the maximal compact extension of $\Zb$ in $\Xb$, so that
we have the diagram $\Xb \to \tilde{\Yb} \to \Yb \to \Zb$. 

Suppose first that the map $\tilde{\Yb} \to \Yb$ is not an isomorphism. 
As this map is a compact extension we can represent $\tilde{\Yb}$ as a skew product over $\Yb$ :
$$
\tilde{\Yb} = \Yb \times_\phi K/H,
$$
with $K$ a compact group and $H < K$ a closed subgroup
and $\phi : Y \to K$ a measurable cocycle.
Then, the extension $\tilde{\Yb} \underset{\Zb} { \times} \tilde{\Yb} 
\to \Yb \underset{\Zb} { \times} \Yb$ is not relatively ergodic.
Indeed, above any  ergodic component $W \subset \Yb \underset{\Zb} { \times} \Yb$,
we have a nontrivial ergodic decomposition which corresponds to the 
ergodic decomposition of the diagonal action of $K$ on  $K/H \times K/H$.
Now, a fortiori, the extension 
$\Xb \underset{\Zb} { \times} \Xb \to \Yb \underset{\Zb} { \times} \Yb$
is not relatively ergodic. 

%Conversely, assuming that 
%the extension $\Xb \underset{\Zb} { \times} \Xb \to \Yb \underset{\Zb} { \times} \Yb$
%is relatively ergodic we have to show that  $\tilde{\Yb} = \Yb$. 
%Now by  \cite[Theorem 6.13]{F-81}
%the Hilbert space $L_2(\tilde{\Yb})$ is generated by functions of the form
%$H * f$, where 
%$$
%H * f(x) = \int H(x, x') f(x') d\, \mu_{\pi(x)}(x'),
%$$
%$H$ ranges over the invariant functions in $L_2(\Xb \underset{\Zb} { \times} \Xb)$,
%and $f \in L_2(\Xb)$, and in the above formula $\pi : X \to Y$ is the factor map,
%and $\mu = \int \mu_y d\, \nu(y)$ is the disintegration of $\mu$ over $\nu$.
%Now our assumption implies that every such $H$ is $L_2(\Yb \underset{\Zb} { \times} \Yb)$,
%and it follows that every $H * f$ is in $L_2(\Yb)$, whence $\tilde{\Yb} = \Yb$.

For the other direction assume that $\tilde{\Yb} = \Yb$ and let $f$ be an invariant 
function in $L_2(\Xb \underset{\Zb} { \times} \Xb)$.
By \cite[Theorem 7.1]{F-77} the function $f$ is a member of the Hilbert space
$L_2(\tilde{\Yb} \underset{\Zb} { \times} \tilde{\Yb}) =
L_2(\Yb \underset{\Zb} { \times} \Yb)$. 
\end{proof}

\br

\br

\begin{lem}\label{succ}
Let $\Wb \to \Xb \to \Yb \to \Zb$ be a tower of extensions of ergodic systems such that
$\Yb \to \Zb$ is the maximal compact extension of $\Zb$ within $\Xb$, and 
$\Xb \to \Yb$ is the maximal compact extension of $\Yb$ within $\Wb$.
Then $\Yb \to \Zb$ is the maximal compact extension of $\Zb$ within $\Wb$.
\end{lem}

\begin{proof}
If the extension $\Yb \to \Zb$ is not the maximal compact extension of $\Zb$ within $\Wb$, then 
there are functions $f_1, f_2, \dots, f_k \in L_2(W)$ such that the finite dimensional $L_\infty(Z)$ module 
$$
L (f_1,\dots,f_k) = \{h_1 f_1 + h_2 f_2 + \dots + h_kf_k : h_i \in L_\infty(Z), i = 1, 2, \dots,k\}
$$
is $\Ga$ invariant, and $L \not\subset L_2(Y)$.
Since $L_\infty(Z) \subset L_\infty(Y)$ it follows that $L$ is also a 
$\Ga$ invariant finite dimensional $L_\infty(Y)$ module, whence, by the maximality of 
$\Xb$ in $\Wb$, we have $L \subset L_2(X)$. Now the maximality of $\Yb$ in $\Xb$
implies that the $L_\infty(Z)$ module $L$ is in fact a subset of $L_2(Y)$ and 
this contradiction proves our claim.
\end{proof}

\br

\begin{lem}\label{IL}
Let $\al$ be a countable limit ordinal.
Let $\Xb$ be an ergodic distal $\Z$-system built as a tower of height $\al$
consisting of group extensions and inverse limits, 
such that for each ordinal $\beta < \al$
the extension $\Xb_{\beta +1}  \to \Xb_{\beta}$ is the maximal compact extension of $\Xb_\beta$ 
within $\Xb_{\beta +2}$, then for each $\beta < \al$ each extension 
$\Xb_{\beta +1}  \to \Xb_{\beta}$ is the maximal compact extension of $\Xb_\beta$ within $\Xb$.  
\end{lem}

\begin{proof}
In view of Lemma \ref{BF} what we have to show is that, for each $\beta < \al$, the extension
$\Xb \underset{\Xb_\beta} {\times} \Xb \to  \Xb_{\beta+1} \underset{\Xb_\beta}{ \times} \Xb_{\beta+1}$
is ergodic. 
Now this extension is an inverse limit of the extensions
$\Xb_\eta  \underset{\Xb_\beta} {\times} \Xb_ \eta  \to  
\Xb_{\beta+1} \underset{\Xb_\beta}{ \times} \Xb_{\beta+1}$,
where the ordinal $\eta$ ranges over the interval $\beta < \eta < \al$.
Now applying transfinite induction,
using Lemma \ref{succ} and the fact that an inverse limit
of ergodic extensions is an ergodic extension,
 we conclude the proof.
\end{proof}

\br

In view of the above three lemmas we conclude that in order to prove Theorem \ref{thm-main}
we only need to prove the following two statements:

\begin{thm}\label{rank-succ}
Let $\al$ be a countable successor ordinal.
Let $\Xb$ be an ergodic distal $\Z$-system of rank $\al$, 
where in the canonical tower the final extension
$\Xb = \Xb_\al  \to \Xb_{\al -1}$ is a compact extension
%(rather than a homogeneous space extension),
and let $G$ be a compact second countable group. 
Then, for a generic function $\phi :  X  \to G$, the corresponding 
skew-product $(X \times G, T_\phi)$ is an ergodic compact extension of $\Xb$ which is 
distal of rank $\al +1$.
\end{thm}

\begin{thm}\label{rank-limit}
Let $\al$ be a countable limit ordinal.
Let $\Xb$ be an ergodic distal $\Z$-system of rank $\al$, 
and let $G$ be a compact second countable group. 
Then, for a generic function $\phi :  X  \to G$, the corresponding 
skew-product $(X \times G, T_\phi)$ is an ergodic compact extension of $\Xb$ which is 
distal of rank $\al +1$.
\end{thm}

\br

\section{A generic cocycle is ergodic}\label{sec-toy}
 
As a warm up let us first prove a simpler statement which generalises a theorem of Jones and Parry
\cite{JP}, where the authors deal with the case where the extending group $G$ is a compact
second countable abelian group.

\br

So for now let $\Yb = (Y, \mathcal{Y}, \nu,T)$ be an ergodic system, 
$G$ a compact second countable topological group, $\la_G$ its normalized Haar measure
and $d_G$ a bi-invariant metric on $G$.
Let $\mathcal{C}= \Ccal(Y,G)$  be the space of Borel maps $\phi : Y \to G$,
where we identify $\phi$ and $\psi$ if they agree $\nu$- a.e.
We equip $\mathcal{C}$ with a metric $d$ as follows:
$$
d(\phi, \psi) =
\inf \{\ep : \nu(\{y \in Y : d_G(\phi(y), \psi(y)) > \ep \}) < \ep\}.
$$

\begin{claim}
The function $d$ above defines a complete metric on the space $\Ccal(Y,G)$.
With the induced topology $\Ccal(Y,G)$ is second countable; i.e. it is a Polish space.
\end{claim}

\begin{proof}
We only check that $d$ satisfies the triangle inequality.
Let $d(\phi, \psi) = \ep_1, d(\psi,\rho) = \ep_2$.
By definition there are sequences $\ep_i \searrow \ep$, $\eta_i \searrow \eta$, such that
\begin{gather*}
\nu(\{ y : d_G(\phi(y), \psi(y)) > \ep_i \}) < \ep_i,\\
\nu(\{ y : d_G(\psi(y), \rho(y)) > \eta_i \}) < \eta_i.
\end{gather*}
Then
\begin{gather*}
\{ y : d_G(\phi(y), \rho(y)) > \ep_i + \eta_i \} \subset\\
\{ y : d_G(\phi(y), \psi(y)) > \ep_i \} + \{ y : d_G(\psi(y), \rho(y)) > \eta_i \},
\end{gather*}
hence
\begin{gather*}
\nu(\{ y : d_G(\phi(y), \rho(y)) > \ep_i + \eta_i \})  \leq \\
\nu(\{ y : d_G(\phi(y), \psi(y)) > \ep_i \}) + \nu(\{ y : d_G(\psi(y), \rho(y)) > \eta_i \}) <\\
 \ep_i + \eta_i.
\end{gather*}
Therefore
$$
d(\phi, \rho)  \leq \inf (\ep_i + \eta_i) = d(\phi,\psi) + d(\psi,\rho).
$$
\end{proof}

Recall that the {\em finite full group} of the system $\Yb$, denoted by $[T]_f$, is defined as the
group of invertible measure preserving transformations 
$\tau = \tau_{\Pcal,\sig}$ of the probability space $(Y, \Ycal, \nu)$
for which there is a finite measurable partition $\Pcal = \{P_1,\dots, P_n\}$ of $Y$
and a function 
$\sig : Y \to \Z$ so that for every $j$,  $\sig \rest P_j = s_j$ is a constant, and
$\tau \rest P_j = T^{s_j}$.

With $\phi \in \Ccal$ we associate the {\em skew product transformation}
$T_\phi : Y \times G \to Y \times G$ which is defined by
$$
T_\phi(y,g) = (Ty, \phi(y)g), \quad y \in Y, g \in G.
$$ 
We then see that
$$
T^n_\phi(y,g) = (T^ny, \phi_n(y)g), 
$$ 
where
$$
\phi_n(y)=\begin{cases}
\phi(T^{n-1} y)\cdots\al(Ty)\phi(y) & {\text{for}}\  n\ge 1\\
\id & {\text{for}}\  n=0\\
\phi(T^n y)^{-1}\cdots\phi(T^{-1} y)^{-1} & {\text{for}}\  n< 0.
\end{cases}
$$
For $\tau = \tau_{\Pcal,\sig} \in [T]_f$ we denote
$$
\phi_\tau(y) = \phi_{\sig(y)}(y), \quad  y \in Y.
$$

\br

\begin{thm}\label{generica}
For a generic $\phi \in \mathcal{C}$ the system 
$\Yb_\phi = (Y \times G, \mathcal{Y} \times \Bcal_G, \nu \times \la_G,T_\phi)$
is ergodic.
\end{thm}

\begin{proof}
Fix a subset $C \subset Y, \ \nu(C) >0$, an element $g$ in $G$, a positive small
constant $a$, and a positive constant $c_a$ (which will depend on $a$). 
Define the set 
$$
U(C, a, c_a, g) \subset \mathcal{C}
$$
as the collection of all the functions 
$\phi \in \mathcal{C}$
with the following property:

\br 

There exists an element $\tau \in [T]_f$ 
such that :
\begin{enumerate}
\item
$ \tau (C) =  C$,
 \item
$$
\nu(\{y \in C : d(\phi_{\tau}(y) , g) < a\}) > c _a \nu(C)
$$
\end{enumerate}

\br

\begin{prop}\label{open-densea}
For a sufficiently small $c_a$ the set $U = U(C,a, c_a, g)$ is open and dense in $\mathcal{C}$.
\end{prop}

\begin{proof}
Fix $C$ and $a$.
It is easy to check that $U$ is an open set. 
In order to see that it is dense
(for a sufficiently small $c_a$)  fix $\phi_0 \in \mathcal{C}$ and $\del >0$ (to be determined later on). 
As finite valued functions are dense in $\mathcal{C}$, we may and will assume that $\phi_0(Y)$
is a finite subset of $G$.

Let $B_0 \subset Y$ be a base for a Rokhlin tower in $Y$ of height $2N +1$, 
so that 
$\nu(\bigcup_{i =0}^{2N} T^i B_0) \allowbreak  > 1 - \del$ (again, $N$ will be determined later on). 
Next purify the $Y$-tower according to $\phi_0$ and $C$; i.e. subdivide $B_0$
into a finite number of subsets $\{B_j\}_{j =1}^J$ so that for each $j$, 
on each level of the column of the tower above 
$B_j$, the function $\phi_0$ is constant, and each level 
is either contained in $C$ or in $Y \setminus C$.

By the ergodic theorem, for sufficiently large $N$, 
but for a set of measure $< \del$, in the remaining
columns of the tower, there is (almost) an
equal proportion of $C$-levels in the bottom half and top half of the tower (up to $\del$).
In each of the good columns enumerate the $C$-levels, from the bottom up to level $N-1$
and from level $N$ to the top. 
Both in the top half and in the bottom half of the tower there are at least 
$(1 - \del)N \nu(C)$ levels contained in $C$ 
($C$-levels). The transformation $\tau$ is defined by interchanging pairs of $C$-levels in the lower and upper
half of the tower, using the appropriate power of $T$. It is defined to be the identity elsewhere.

When we focus on a particular $C$-level, say $i_0$ in the lower half,
that maps to a $C$-level $i'_0$ in the upper half of the column,
the value of the cocycle  corresponding to $\phi_0$, as we move up the tower to level $N-1$, 
and the value of the cocycle as we move from level $N+1$ up to the level $i'_0$, define a pair
$(h_1, h_2) \in G \times G$.
As $\phi_0$ has only finitely many values the collection $H$ of all pairs $(h, h') \in G \times G$ 
so obtained is finite. 
Thus $\tau$ moves from the bottom to the top at least a $(\frac12 - 10 \del)$-fraction of $C$.

\br

Let $V$ be $\frac {a}{10}$-ball around $e$.
Let $F = \{f_1, \dots, f_m\} \subset G$ be a finite set with $G = \bigcup_{i=1}^m f_iV$.
As our metric on $G$ is bi-invariant, for any pair 
$(h,h') \in H$ there is an element $f \in F$ with $hfh' \in V$.
Note that $m = m(a)$ depends only on $a$.
Now divide each central level $T^N B_j$ into $m$ sets $\{D^j_1, \dots, D^j_m\}$ of equal measure, 
and change $\phi_0$ to $\phi$ by defining $\phi \rest D^j_l = f_l,  \ 1 \leq l \leq m$. 
%We  can now check and see that, with $c_a = \frac1m (\frac12 - 10 \del)$, we have
%$$
%\mu(\{y \in C : d(\phi_{\tau}(y) , g) < a\}) > c _a \mu(C).
%$$
For the pair $(h, h')$ in the previous paragraph there is at least one value of $l$ such that
$$
d(hf_l h', g) < a
$$
and thus, for at least a $\frac1m$-fraction of the $i_0$ level we see that $d(\phi_\tau(y),g) < a$.
Thus with $c_a = \frac1m (\frac12 - 10 \del)$ we have
$$
\nu(\{y \in C : d(\phi_{\tau}(y) , g) < a\}) > c _a \nu(C).
$$
\end{proof}

\br

In order to finish the prof of Theorem \ref{generica}
we need to show that for a $\phi$ in a dense $G_\del$ subset of $\Ccal$ the corresponding system
$\Yb_\phi = (Y \times G, \mathcal{Y} \times \Bcal_G, \nu \times \la_G,T_\phi)$ 
is ergodic.

Given $C, a$ and $g$ as above, we define the set $U= U(C,a, c_a, g) \subset \mathcal{C}$ 
and by Proposition \ref{open-densea} we know that, with a suitable constant $c_a$, 
it is open and dense in $\Ccal$.

Let  $\{C_n\}_{n \in \N}$ be a dense collection in the measure algebra $(\Ycal, \nu)$ and set
$$
\Ucal(a, c_a, g) = \bigcap_{n \in \N}  U(C_n, a, c_a,g).
$$
Note that it is here that we use the uniform bound 
$\nu(\{y \in C : d(\phi_{\tau}(y) , g) < a\}) > c _a \nu(C)$, for now we have that 
$\phi \in \Ucal(a, c_a, g)$ implies $\phi \in U(C, a, c_a, g)$ for every 
positive $C \in \Ycal$.

Next set
$$
\Ucal(g) = \bigcap_{n \in \N} \Ucal(\frac{1}{n}, c_{\frac{1}{n}}, g),
$$
and finally let 
$$
\Ccal_0  = \bigcap \{\Ucal(g) : g  \in G_0\},
$$
with $G_0 \subset G$ a countable dense subset of $G$.

We now see that every $g \in G$ is an essential value for each $\phi$ 
in the dense $G_\del$ subset  $\Ccal_0$ of $\Ccal$. This is a sufficient (and necessary) condition for $T_\phi$
to be ergodic (see e.g. \cite[Page 1287]{AW}) and the proof of Theorem \ref{generica} is complete. 
\end{proof}

\br

\section{The proofs of  Theorems \ref{rank-succ} and  \ref{rank-limit}}\label{sec-ranks}

We deal first with Theorem \ref{rank-limit}.
Thus, our system $\Xb$ is assumed to be ergodic and distal of order $\al$, with $\al$ a limit ordinal.
As the distal tower for $\Xb$ is canonical, for each $\beta < \al$,
the extension $\Xb_{\beta + 1} \to \Xb_\beta$ is the maximal compact extension of $\Xb_\beta$ in $\Xb$.
In order to show that, for a generic $\phi : X \to  G$, the system 
 $(X \times G, T_\phi)$, denoted by $\Xb_\phi$, is ergodic of rank $\al +1$,
what we have to show,
in view of Lemma \ref{IL}, is that for each ordinal $\beta < \al$, 
the extension
$$
\Xb_\phi \underset {\Xb_\beta} {\times} \Xb_\phi \to \Xb_{\beta +1} \underset {\Xb_\beta} {\times} \Xb_{\beta + 1}
$$
is relatively ergodic.
This situation is isolated in the following setup, where we use the notation
$\Zb = \Xb_\beta$ and $\Yb = \Xb_{\beta +1}$.

\begin{setup}\label{setup-limit}
Let $K$ and $G$ be two arbitrary second countable compact topological groups
with normalized Haar measures $\la_K$ and $\la_G$ respectively.
Let $H \leq K$ be a closed subgroup, with its Haar measure $\la_H$.
We write $\la_{K/H}$ for the normalized Haar measure on the homogeneous space $K/H$.
 Suppose that $\Zb = (Z, \mathcal{Z}, \theta, T_Z)$ is an ergodic $\Z$-system 
and that $\Yb =(Y,\Ycal, \nu, T_Y)$ is the skew product 
$$
\Yb\cong \Zb \times_\ga (K/H,\la_{K/H})
=(Z\times K/H,\Zcal\otimes\Bcal_{K/H},\theta\times\la_{K/H},T),
$$
where $Y = Z \times K/H$, $\nu = \theta \times \la_{K/H}$ and 
$T_Y = (T_Z)_\ga$ for a cocycle $\ga : Z \to K$ which is minimal with $K_\ga= K$, 
as in Theorem \ref{homog},
and such that the system $\Yb$ is ergodic.

%--------

\br

We recall that the corresponding group extension $\hat{\pi} : \hat{\Yb} \to \Zb$,
with $\hat{Y} = Z \times K$, is ergodic and the diagram
\begin{equation*}
\xymatrix
{
\hat \Yb =\Zb  \times_\ga K
 \ar[d]_{\hat\pi} \ar[dr]^{\sig}  & \\
\Zb &  \Yb=\Yb  \times_\ga  K/H \ar[l]^-{\pi}
}
\end{equation*}
commutes, 
where $\hat\Yb = \Zb  \times_\ga (K, \la_K)$
is the group skew-product defined by the cocycle $\ga$,\
with $\hat\nu=\theta\times \la_K$ 
and $\hat \nu$ is ergodic. The map $\sig : \hat{\Yb} \to \Yb$ 
is the quotient map $\sig(y,k) = (y, kH)$.

\br

Suppose further that $\Xb = (X, \Xcal, \mu, T_X)$ is an ergodic system,
$$
\Xb = (X, \Xcal, \mu, T_X) \to \Yb = (Y, \mathcal{Y}, \nu,T_Y)
$$ 
a factor map, such that 
in the diagram $\Xb \to \Yb \to \Zb$, the map
$\Yb \to \Zb$ is the maximal compact extension of $\Zb$ in $\Xb$.
 By Rokhlin's theorem \ref{rsp} we can represent $X$ as $Y \times V$ where $(V, \rho)$ is 
 a probability space, $\mu = \nu \times \rho$ and the transformation $T_X : X \to X$ is given by a Rokhlin cocycle
 $S$ from $Y$ to the Polish group $MPT(V, \rho)$ of invertible measure preserving transformations
 of $(V, \rho)$, with
 $$
T_X(y,v) = (T_Yy, S_y v), \quad y \in Y, v \in V.
 $$
%(Note that we use $T$ for both $\Xb$ and $\Yb$.)

We let let $\mathcal{C} = \Ccal(X,G)$ 
be the Polish space of Borel maps $\phi : X \to G$.
For $\phi \in \mathcal{C}$ set $T_\phi : X \times G \to X \times G$ as :
$$
T_\phi(x, g) = (T_Xx, \phi(x)g), \quad (x \in X, g \in G).
$$
We let $\Xb_\phi =(X \times G, \Xcal \times \Bcal_G, \mu \times  \la_G)$ 
be the corresponding skew product system.

\br

We need to show that, for a generic $\phi \in \Ccal$ the system 
$\Xb_\phi \underset{\Zb}{\times} \Xb_\phi$ is a relatively ergodic extension of $\Yb \underset{\Zb}{\times} \Yb$.
The ergodic components of $\Yb \underset{\Zb}{\times} \Yb$ are parametrized by $k_0H  \in K/H$
and may be identified with $\Zb \times H/H_{k_0}$,
with $H_{k_0} =k_0 H k_0^{-1} \cap H$, as follows.

\br 

The ergodic components of the system $\hat{\Yb} \underset{\Zb}{\times} \hat{\Yb}$
are parametrized by $k_0 \in K$ and have the form $\{(z, k, kk_0) : z \in Z, \ k \in K\}$.
Under $\sig$ these are mapped onto the sets
$$
\{(z,kH, kk_0H) :  z \in Z, \ k \in K\}.
$$

Now for $k_1, k_2 \in K$ we have 
$(z,k_1H, k_1k_0H) = (z,k_2H, k_2k_0H)$ iff
$$
k_1^{-1}k_2 \in H \cap  k_0H k_0^{-1} = H_{k_0},
$$ 
and, in particular, $(z,kH, k k_0H) = (z, H,k_0H)$ iff $k \in H_{k_0}$.
Thus $H_{k_0}$ is the stability group of the point $(H, k_0H)$  under the left action by $K$.
It therefore follows that the correspondence
\begin{equation}\label{Hk0}
(z,kH, kk_0H)  \longleftrightarrow (z, kH_{k_0})
\end{equation}
is an isomorphism, with the diagonal action:
$$
(T_Y \times T_Y)(z,kH, kk_0H) = (T_Zz, \ga(z)kH, \ga(z) kk_0H) 
 \longleftrightarrow (T_Zz, \ga(z) kH_{k_0}).
$$

{\bf In the sequel, in order to simplify our notations,
we will often write $T$ instead of $T_X$, $T_Y$ and $T_Z$,
when the meaning is clear from the context. Whenever there is room for confusion
we will revert to $T_X$ etc.}

\br

Our assumption on the diagram 
$\Xb \to \Yb \to \Zb$ implies that 
$\Xb \underset{\Zb}{\times} \Xb$ is a relatively ergodic extension of 
$\Yb \underset{\Zb}{\times} \Yb$, and thus for a fixed $k_0$, 
the diagonal transformation $T_X \times T_X$  on $X \underset{Z}{\times}X$
%\begin{gather*}
%T_{k_0} : (z, kH, kk_0H, v_1, v_2)
%%=(y, L{k_0,h} y, v_1, v_2) 
%\mapsto \\
%(Tz, \ga(z) kH,  \ga(z) kk_0H, S_{(z,kH)} v_1, S_{(z, khk_0H)}v_2),
%\end{gather*}
%or more briefly
\begin{equation}\label{Tk0}
T_{k_0}(z, kH,  kk_0H, v_1, v_2) =
(Tz, \ga(z) kH, \ga(z) kk_0H, S_{(z,kH)} v_1, S_{(z, kk_0H)}v_2),
\end{equation}
is ergodic.
We note that all these sets (the supports of the ergodic components) are subsets of the space
\begin{equation*}\label{space}
Z \times K/H \times K/H \times V \times V.
\end{equation*}

\br
Alternatively, under the correspondence (\ref{Hk0})
\begin{equation*}\label{Tk}
%T_{k_0} (y, L_{k_0,h} y, v_1, v_2) = (Ty,  TL_{k_0,h}y, S_yv_1, S_{L_{k_0,h} y}v_2).
T_{k_0} (z, kH_{k_0}, v_1, v_2) = (Tz,  \ga(z)kH_{k_0} , S_{(z,kH)} v_1, S_{(z, kk_0H)}v_2).
\end{equation*}
is ergodic.
Denoting $L_{k_0} = K/H_{k_0}$, we note that this identifies the ergodic component associated with
$k_0$ with $Z \times L_{k_0} \times V \times V$. 
We write $\mu_{k_0}$ for the product measure
$\theta \times \la_{L_{k_0}} \times \rho \times \rho$ on this space.

We will use the same letter $\mu_{k_0}$ for the corresponding measure on the space 
%(\ref{space}).
$X \underset{Z}{\times}X$.
Thus in this notation we have
$$
\theta \times \la_{K/H} \times \la_{K/H}  \times \rho \times \rho =
\int_{K/H} \mu_k \, d \la_{K/H}(k).
$$
{\bf Note that with this representation the transformations $T_{k_0}$, as in (\ref{Tk0}), 
are all one and the same as the diagonal action of $T= T_X$ on $X\underset{Z}{\times}X$,
and we only change the measure, so that $T_k = T_X \times T_X$ is ergodic with respect to $\mu_k$.}

\br

What we need to show is that for a generic cocycle $\phi : X \to G$, the cocycle $\psi = \psi_{k_0}$, 
defined from 
%$Z \times L_{k_0} \times V \times V$ 
$Z \times K/H \times K/H \times V \times V$ to $G \times G$ by
%$$
%\psi(z, kH_{k_0}, v_1,v_2) = (\phi(z, kH,v_1), \phi(z, kk_0H,v_2))
%$$
\begin{equation}\label{psi}
\psi(z, kH,  kk_0H, v_1,v_2) = (\phi(z, kH,v_1), \phi(z, kk_0H,v_2)),
\end{equation}
is ergodic with respect to $\mu_{k_0}$ for  $\la_{K/H}$ a.e. $k_0H \in K/H$.
\end{setup} 

%
%As compared to our earlier argument in Section \ref{rank-succ},
%the situation is now complicated by the fact that on each ergodic component (defined by $k_0$) of 
%$\Xb \underset{Z}{\times} \Xb$ there is a different transformation, since on
%$V \times V$ the Rokhlin cocycle over $y \in Y$ is now $S_y \times S_{yk_0}$.
%
%None the less we will follow in the main the steps of the proofs of Theorem \ref{max} and  
%of Proposition \ref{open-dense}.

\br

\begin{rmk} \label{finite}
Suppose $\Yb \to \Zb$ is a $K$ extension with $K$ a finite group.
Then as we noted in the above discussion, 
the ergodic components of the system $\Yb \underset{\Zb}{\times} \Yb$
are parametrized by $k_0 \in K$ and have the form $\{(z, k, kk_0) : z \in Z, \ k \in K\}$.
It then follows that the ergodic component
which corresponds to $k_0 =e$, i.e. the component $W = \{(z, k, k) : z \in Z, \ k \in K\}$,
with diagonal action $T(z, k, k) = (Tz, \ga(z) k, \ga(z)k)$
has positive measure in the system $\Yb \underset{\Zb}{\times} \Yb$.
But then, no matter which $\phi \in \Ccal(Y,G)$ we choose, the extension $\Yb_\phi \to \Yb$
can not be such that 
$$
 \Yb_\phi \underset{\Zb} { \times} \Yb_\phi  \to \Yb \underset{\Zb} { \times} \Yb.
$$
is ergodic since that, in particular, means that the ergodic component
in the system  $\Yb_\phi \underset{\Zb} { \times} \Yb_\phi$ which sits above $W$
with diagonal action
\begin{align*}
T_\phi(z, k, k, g_1, g_2) & = (Tz, \ga(z) k, \ga(z)k, \phi(z,k) g_1, \phi(z,k) g_2)\\
&  = (Tz, \ga(z) k, \ga(z)k, \phi(z,k), \phi(z,k) g_2( g_1)^{-1})g_1,
\end{align*}
is $G \times G$ saturated, which is impossible.

This simple observation shows why in our Theorem \ref{thm-main},
as well as the other main results in this paper, we have to require that the extending compact
groups (or homogeneous quotient) be infinite.
 \end{rmk}

\br

\begin{rmk} \label{rmk2}
The next theorem is the main tool in the proof of Theorem \ref{rank-limit}.
However, we note that whereas in Theorem \ref{rank-limit} we assume that
 the system $\Xb$ is distal, here we only need the system $\Zb$ to  be ergodic and infinite. 
 \end{rmk}

\begin{thm}\label{max2}
We are given a chain of factors $\Xb \to \Yb \to \Zb$ of the ergodic system $\Xb$ such that
$\Zb$ is infinite, $\Yb \to \Zb$ is a $K/H$-extension, and
$\Yb \to \Zb$ is the maximal compact extension of $\Zb$ in $\Xb$.
Then, for a generic $\phi \in \mathcal{C}(X,G)$ 
the system $\Xb_\phi = (X \times G, \mathcal{X} \times \Bcal_G, \mu \times \la_G, T_\phi)$,
is ergodic and the extension, $\Yb \to \Zb$, 
is the maximal compact extension of $\Zb$ in $\Xb_\phi$.
\end{thm}

\begin{proof}

\br

Fix a subset $C \subset Z \times K/H \times K/H \times V \times V$
of positive measure,  elements $g_1, g_2$ in $G$, positive small
constants $a, b$, and a positive constant $c_a$ (which will depend on $a$). 
Define the set 
$$
U = U(C,a,b, c_a, g_1, g_2) \subset \Ccal(X,G),
$$
%where $C$ is a given measurable subset of $Z \times K/H \times K/H \times V \times V$, 
%and $U$ is 
to be the set of cocycles 
$\phi : Z \times K/H \times V \to G$ such that for a set $K_0 \subset K$ with $\la_K(K_0) > 1 - b$,
and all $k_0 \in K_0$,
there is an element $\tau \in [T_X]_f$ such that
\begin{align}\label{ess1}
\begin{split}
\mu_{k_0} (\{(z, kH,  kk_0H, v_1,v_2) \in C : & \  d(\phi_\tau(z, kH,v_1) , g_1) < a \quad   \& \  \\  
& 
d(\phi_\tau(z, kk_0H,v_2) , g_2) < a\})  > c_a  \mu_{k_0}(C)^2.
\end{split}
\end{align}

\br

\begin{prop}\label{open-dense2}
For a sufficiently small $c_a$ the set $U(C,a, b, c_a, g_1, g_2)$ is open and dense in $\mathcal{C}(X, G)$.
\end{prop}

\begin{proof}
%The fact that $U=U(C,a, b, c_a, g_1, g_2)$ is open is proved,  almost verbatim, in the same way as in 
%Proposition \ref{open-dense}.
%
%
%%-----------------
%
%
Fix $C, a, b$ and $g_1, g_2$ and consider the corresponding set $U = U(C,a, b, c_a, g_1, g_2)$.
%It is easy to check that $U$ is an open set. 
%In fact, given $\phi_0 \in U$, for each $n$ let
% \begin{align*}
%\begin{split}
%K_n  = \{k_0 \in K_0   : &
%\mu_{k_0} (\{(z, kH,  kk_0H, v_1,v_2) \in C :  \  d((\phi_0)_\tau(z, kH,v_1) , g_1) < a \quad  \& \\  
%& d((\phi_0)_\tau(z, kk_0H,v_2) , g_2) < a\})  > c_a  \mu_{k_0}(C)^2  + \frac1n\}.
%\end{split}
%\end{align*}
%%
%%
%%\begin{align*}
%%K_n  = \{k \in K_0 : & 
%%\tilde{\mu} (\{(y, v_1, v_2) \in C : d(\phi_\tau(y, v_1) , g_1) < a \ \& \ 
%%d(\phi_\tau(yk_0, v_2) , g_2) < a\}) \\
%%& > c_a \tilde{\mu}(C)^2  + \frac1n\}.
%%\end{align*}
%%$$
%%K_n = \{k \in K_0 : 
%%\mu(\{y \in C : d({\phi_0}_{\tau}(y) , g_1) < a \   \&  \  d({\phi_0}_{\tau}(yk), g_2) < a \}) > c _a \mu(C) + \frac1n\}.
%%$$
%% and let $\la_K(K_0) = 1 - b +\ep$.
%Since $K_0 = \bigcup_{n \in \N} K_n$, for some $n_0$, $\la_K(K_{n_0}) > 1- b$.
%If $\phi \in \Ccal$ is such that
%$$
%\mu_{k_0} (\{(z, kH,  kk_0H, v_1,v_2) \in C : \phi_0(z, kH, v_1) \not = \phi(z, kk_0H, v_2)\}) < \frac{1}{2n_0},
%$$
%then for $k_0 \in K_{n_0}$
%\begin{align*}
%\mu_{k_0} (\{(z, kH,  kk_0H, v_1,v_2)  \in C :  & d(\phi_{\tau}(z, kH, v_1) , g_1) < a
% \   \&  \\  
% & d(\phi_{\tau}(z, kk_0H, v_2), g_2) < a \}) > c _a \mu_{k_0}(C)^2 + \frac{1}{2n_0},
%\end{align*}
%and therefore $\phi \in U$.
%

%--------------------------

In order to see that $U$ is open, note that all of the measure inequalities are strict
and for a fixed $\phi_0 \in U$ the corresponding element $\tau_0 \in [T_X]_f$ involves only finite
number of products of $\phi_0$.
It is therefore clear that for $\del$ sufficiently small, if $\phi$
satisfies$$
\mu(\{x : d(\phi_0(x), \phi(x) > \del\}) < \del
$$
then $\phi$ will also belong to $U$.

\br

In order to see that $U$ is dense
(for a sufficiently small $c_a$)  fix $\phi_0 \in \mathcal{C}$
and $\del >0$. We will show that there is a $\phi \in U$
with $\mu(\{ (y,v)  \in Y \times V : \phi(y, v) \not = \phi_0(y, v) \}) < \del$.
As finite valued functions are dense in $\mathcal{C}$, 
we may and will assume that $\phi_0(Y \times V)$ is a finite subset of $G$.
We denote by $\psi_0$ the corresponding (family of) cocycles (\ref{psi}).
\br

{\bf Step 1: Constructing a Rokhlin tower} 

\br

In $X \underset{Z}{\times}X \cong Z \times K/H \times K/H  \times V \times V$ 
we will take a Rokhlin tower with base 
$B = B_0 \times K/H \times K/H \times V \times V$ measurable with respect
to $Z $, which is a factor of all the ergodic components $\mu_k$, of height $2N +1$
with $N > \frac{1}{10\del}$, 
so that, for every $k \in K$ we have $\mu_k(\bigcup_{j =0}^{2N} T_k^n B) > 1 - \frac{\del}{100}$. 
Further conditions on how large $N$ should be will be imposed later.
%No matter which $T_k$ we are using the levels of this tower are of the form 
%$(T^i B_0) \times K \times V \times V$.
%As before we purify this tower with respect to both $C$ and $\phi_0$.

Next purify the $Z \times K/H \times K/H  \times V \times V$-tower according to $\psi_0$ and $C$; 
i.e. subdivide $B$
into a finite number of subsets $\{B_j\}_{j=1}^J$ so that for each $j$, 
on each level of the column of the tower above $B_j$,
the function $\psi_0$ is constant, and each level 
is either contained in $C$ or in $Y \setminus C$. Of course the atoms of the purified tower
are no longer $K/H \times K/H \times V \times V$-saturated. We call the columns $\{T_k^i B_j\}_{i=0}^{2N},\
0 \leq j  \leq 2N$, the {\em pure columns}.

Now for almost every $k \in K$ the transformation $T_k$ is ergodic and hence, for some $n_0(k)$,
if $N  \geq  n_0(k)$ we will have that, having purified the tower with respect to $C$, 
in most of the pure columns there will be
a proportion of roughly $\mu_k(C) N$ of the levels contained in $C$ in both the lower and 
upper half of the tower. Since the family $\{T_k : k \in K\}$ is measurable with respect to $k$, the index $n_0(k)$
is a measurable function of $k$ and consequently there will be a set $K_0 \subset K$,
with $\la_K(K_0) > 1 - b/10$, for which this will hold for a fixed $N$ and for all $k \in K_0$.
%Note that this purification, and therefore also the sets $B_j \subset B$ depend on $k$.

%
%Another complication arises now from the fact that our given cocycle $\phi_0$ is defined
%on the space $X = Y \times V$, whereas in the proof of Theorem \ref{max} it was defined on $Y = Z \times K$.
%As a result we must define our cocycle $\phi$, which we want to be close to $\phi_0$, as a function
%on $Y \times V$. We cannot let the element $\tau \in [T]_f$ depend on $k$.

While we are working here with $X \underset{Z}{\times}X$, the element $\tau$ that we define has to
be in $[T_X]_f$ and thus to be defined on $X$. 
We can not simply assign $C$-levels in the lower half to $C$-levels in the upper half
as we did in the proof of Proposition \ref{open-densea}.
Instead we shall show that a random assignment of the levels in the lower half to levels in 
the upper half will work provided that $N$ is sufficiently large.
Since the tower is measurable with respect to $Z$, such
an assignment will certainly define an element of $[T_X]_f$.

In order to define $\tau$ we will use a random permutation on $\{0,1, \dots,N-1\}$
to interchange the first $N$-levels of the tower with the last $N$ levels as follows.
Let $\Om$ be the sample space of such a random permutation; that is
$\Om = \Sym(N)$ with the uniform measure, and we use $\pi \in \Om$ to define an 
involution $\tau$ between 
$$ 
\bigcup_{j =0}^{N-1} T_k^j(B) \quad
{\text{ and}} \quad
\bigcup_{j =N+1}^{2N} T_k^j(B)
$$ 
by mapping
$T_k^j(B)$ to $T_k^{\pi(j) + N + 1}(B)$ via $T_k^{N+1 -j +\pi(j)}$.
{\bf Restricting to $X = Z \times K/H \times V$,
this $\tau$ is clearly in $[T_X]_f$.}

\br

Set $\ga_k = \mu_k(C)$.
We will show that there is a choice of a random permutation such that the corresponding $\tau$
will map 
a $\ga_k$-proportion of $C \cap (\bigcup_{j =0}^{N-1} T_k^j(B))$ 
to $C$, for $k$ in a set $K_1 \subset K_0$ 
%of $k \in K_1$ 
with $\la_K(K_1) > 1 - b/20$.
To do this we need two lemmas.
For convenience we relegate their proofs to an appendix. 

\br

\begin{lem}\label{randomp}
For any $0 < \ga <1$ 
$$
\Prob \{\pi \in \Sym(N) : | \{i \le \lfloor \ga N \rfloor : \pi(i) \le   \lfloor \ga N \rfloor \} | > \frac12 \ga^2 N \} 
> 1 - \frac{10(1-\ga)}{N \ga^2}.
$$
\end{lem}

\begin{lem}\label{del}
Let $\Pcal = \{P_i : i \in I\}$ be a finite partition of a probability space $(\Om,\mu)$.
Let $E \subset \Om$ with $\mu(E)  < \del$,
then for
$$
I_0 = \{i : \frac{\mu(P_i \cap E)}{\mu(P_i)} > \sqrt{\del}\},
$$
we have then $\sum_{i \not \in I_0} \mu(P_i) > 1 - \sqrt{\del}$.
\end{lem}

\br

\br

%{\bf Step 4: Choosing a good $\tau$ for $\phi_0$} 
{\bf Step 2: Elements $\tau^\om$ good for $\phi_0$}

\br

We can now apply Lemma \ref{randomp} to each fixed pure column based on $B_j$, where there are
at least a $\frac{\ga_{k_0}}{2}$-proportion of the $C$-level in both the upper and lower halfs of the column.
For such pure column, which will fill most of the tower, with probability $\geq 1 - O(\frac{1}{N})$ the random 
$\tau^\om$ will map at least a $\frac{1}{3}\ga_{k_0}^2$-proportion of $C$ to $C$, in that pure column.

In the product space 
%$\Om \times B_0 \times V \times V$ 
$\Om \times B$
define the set $E$ to be the set of pairs
%$(\om, (y, v_1, v_2))$ 
$$
(\om, (z, kH, kk_0H, v_1, v_2))
$$
such that $\tau^\om$ fails the requirement above in the fiber
%$\{T_k^i(y, v_i, v_2)\}_{i=0}^{2N}$, 
$\{T_{k_0}^i(z, kH, kk_0H, v_i, v_2)\}_{i=0}^{2N}$, 
and normalize the measure on $B$ to be a probability measure.
Then we apply Lemma \ref{del}, with $\del = O(\frac{1}{N})$ and $\Pcal$-the partition of $\Om$ into points,
to get a set $\Om_0 \subset \Om$ with
probability $\geq 1 - O(\frac{1}{\sqrt{N}})$ such that for each $\om \in \Om_0$ the random $\tau^\om$ maps
$\frac{1}{3}\ga_{k_0}^2$-proportion $C$ to $C$, but for a $O(\frac{1}{\sqrt{N}})$-proportion of the pure columns.

This was done for a fixed $k_0 \in K$, i.e. the set $\Om_0 \subset \Om$ depends on $k_0$, since the
partition of $B$ into pure columns depends on $k_0$.
Now we take this $O(\frac{1}{\sqrt{N}})$ to be less than $\frac{b}{100}$ and appling Fubini's theorem
to $\Om \times K$, we see that there is a random $\om$ and a set $K_2 \subset K$ with 
$\la_K(K_2) > 1 - \frac{b}{20}$, such that $\tau =\tau^\om$ is good for all $k \in K_2$.
%have a proportion of $\ga := \tilde{\mu}(C) N$ levels in $C$, and thus with high probability $\tau$ will map
%a subset of $C$ of measure $> \ga^2/2$ to $C$.
%If we weight each pure column by $\tilde{\mu}(B_j)$, and apply Lemma \ref{del}, we see that with probability
%of error $O(\frac{1}{\sqrt{N}})$ a random $\tau$ will map a subset of $C$ with measure at least
%$\frac{1}{3} \tilde{\mu}(C)^2$ to $C$.
%
%All of this was done for a fixed $k \in K$.
%If we take this error to be $< \frac{1}{100} b$ then, as in the final calculation in the proof of
%Proposition \ref{open-dense}, we see that there is a random $\tau$ which will do this for a set
%$K_2 \subset K$ with $\la_K(K_2) > 1 - \frac{b}{20}$.
We let $K_0 = K_1 \cap K_2$ (where $K_1$ is the subset of $K$ defined in Step 1)
so that $\la_K(K_0) > 1 - \frac{b}{10}$.
{\bf In the rest of the proof we fix such a good $\tau$.}

\br

When we focus on a particular $C$-level, say $i_0$ in the lower half
that maps to a $C$-level $i'_0$ in the upper half of the column,
the value of the cocycle  corresponding to $\phi_0$, as we move up the tower to level $N-1$, 
and the value of the cocycle as we move from level $N+1$ up to the level $i'_0$, define a pair
$(h_1, h_2) \in G \times G$.
%For $(z, kH, kk_0H, v_1,v_2)$ in the bottom $C$-level $i_0$ we write 
%$(h(z, kH, kk_0H, v_1,v_2), h'(z, kH, kk_0H, v_1,v_2))$ for this pair.

Let $V_1, V_2$ be $\frac a{10}$-balls around $g_1, g_2$ respectively.
As in the proof of Proposition \ref{open-densea} there is 
a finite set $F \subset G$ of cardinality $m$ such that for any two pairs 
$(h_1,h'_1)$ and $(h_2,h'_2)$ in $G \times G$ as above there is a pair $(f_1, f_2) \in F \times F$ with
$$
h_1f_1h'_1 \in V_1, \qquad  h_2f_2h'_2 \in V_2.
$$
Note that $m$ depends only on $a$.

{\bf In the next step we will use the set $F$ in order to construct our random cocycles $\phi$
which will modify $\phi_0$ only on the central $N+1$ level of the tower.}

\br

{\bf Step 3: Defining the cocycle $\phi$} 

\br

On the space $X \underset{Z}{\times}X  \cong Z \times K/H \times K/H \times V \times V$
we now define sequences of partitions as follows.
We  choose a refining sequence of measurable partitions $\{\Pcal_i\}$ on $Z$
which together generate the Borel $\sig$-algebra $\Bcal_Z$.
Also,  a refining sequence of measurable partitions $\{\Qcal_i\}$ on $K/H$
which together generate the Borel $\sig$-algebra $\Bcal_{K/H}$.
%Next we choose a sequence of partitions $\{\Qcal_i\}$ of $K$ as in Step 5 and
%of the proof of Proposition \ref{open-dense}.
Finally we let $\{\Rcal_i\}$ be a sequence of refining partitions of $V$
which together generate the $\sig$-algebra $\Bcal_V$.
Now consider the corresponding
sequence of partitions $\{\Pcal_i \times \Qcal_i \times \Qcal_i \times \Rcal_i \times \Rcal_i\}$ of 
$Z \times K/H \times K/H \times V \times V$.
By choosing $l \geq i_0$ sufficiently large we can assume
%, by Lemma \ref{approx}, 
that each atom 
$T_k^N  B_j$
is approximated by $\Pcal_l \times \Qcal_l \times \Qcal_l \times \Rcal_l \times \Rcal_l$ up to a set of relative measure 
$ < \frac{\del}{10}$ (i.e. relative to the total measure of $T_k^N(B_j))$.

We let $d_K$ be a bi-invariant metric on $K$ and let $\bar{d}$ be the restriction of the Hausdorff metric on
the hyperspace $2^K$ of closed subsets of $K$ to the subspace $\{kH : k \in K\}$.
Let $\eta >0$  be such that for $\tilde K = \{k \in K : d_K(k,H) >\eta\}$ we have
$\la(\tilde K) > 1-b/10$,
%\color{red}
(this is possible when $\la_K(H) =0$ ( iff $[K, H] =\infty$)). 
%}
We also assume that $l$ is sufficiently large so that 
$$
\max \{\diam_{\bar{d}}(Q) : Q \in \Qcal_l\} < \eta.
$$

Note that, as the metric $d_K$ is invariant, if $k_0 \in \tilde K$ then 
for any $Q \in \Qcal$ and any $kH \in Q$, we have $\bar{d}(kH, kk_0H) > \eta$,
whence $Q \cap  Q' = \emptyset$, for any $Q'$ that contains an element of
%$Q' \in R_{k_0} (Q)$, where 
$R_{k_0}(kH)= \{\{khk_0H\} : h \in H\}$ for some $kH \in Q$.
%
%$Qk \cap Q = \emptyset$
%for all $Q \in \Qcal_l$. 

%the relative measure of each atom of $\{\Pcal_i \times \Qcal_i\}$ is $< \eta$ (to be determined later).
Once $l$ is chosen we 
%denote by $L$ the number of atoms of $\Qcal_l$. We then
choose  a suitable index $l'$ so that 
\begin{equation}\label{l'}
\max \{\mu(P) : P \in \Pcal_{l'}\}
\end{equation}
is small (how small it needs to be will be determined in Step 4).

We can now define our new cocycle $\phi$ by setting it equal to $\phi_0$ outside 
$T^N(B)$,
and on each atom $P \times Q \times R$ of $\Pcal_{l'} \times \Qcal_l \times \Rcal_l
\rest T^N(B_0 \times K/H \times V)$, 
\begin{quote}
letting $\phi$ be a random variable $\xi^\om_{P \times Q \times R}$
taking values in $F$,  uniformly and independently for distinct atoms, 
so that now $\phi = \phi^\om, \ \om \in \Om$, the sample space of 
all these random variables (not to be confused with the probability space $\Om = \Sym(N)$
defined in Step 1).
\end{quote}
 
 Note that for $k_0 \in \tilde K$, if 
% $(y,v_1)$ and $(yk,v_2)$ 
 $(z, kH, v_1)$ and $(z, kk_0H,v_2)$
 are in distinct atoms of $\Pcal_{l'} \times \Qcal_l \times \Rcal_l$, then
 $\phi^\om(z, kH, v_1)$ and $\phi^\om(z, kk_0H,v_2)$ are independent and therefore take on each value
 of $F \times F$ with probability $1/ |F|^2$.
 
% \br
% 
% {\bf Step 6:  A simple probabilistic lemma} 
%
%\br
%
%Here we prove a probabilistic lemma which will serve us in Step 7.
%
%
% 
%\begin{lem}\label{simple}
%Let $p$ be a number in $(0,1)$ and let $L$ be a positive constant.
%Suppose that $w_j \geq 0$ and $\sum_{j =1}^n w_j =1$.
%Let $X_j, \ j=1,\dots,n$ be random variables taking values in $\{0,1\}$ with $P(X_j =1) \geq p$.
%Suppose further that for each $j$ there is a set $I_j$, \  $ |I_j|  \leq L$,
%such that for $i \not\in I_j$, $X_i$ and $X_j$ are independent.
%Set $w = \max_{1 \leq j \leq n} w_j$.
%Then for $X = \sum_{j=1}^n w_j X_j$ we have
%$$
%P(X \geq p/2)  \geq 1 - \frac{4(L+1)}{p^2} w.
%$$
%\end{lem}
%
%\begin{proof}
%Let $\bar{x}_j = P(X-j =1)$ and define $Y_j = X_j - \bar{x}_j 
%$.
%Let $\bar{p} = \sum_{j=1}^n w_j \bar{x}_j \geq p$.
%Then $\sum w_j Y_j = X - \bar{p}$, and
%if $j \not \in I_j$ we have
%$$
%\E(Y_j Y_i) =0, \quad {\text{while}} \quad \E(X - p) =0.
%$$
%Now
%$$
%(X - \bar{p})^2 = \sum_{j=1}^n w_j^2 Y_j^2 + \sum_{j=1}^n \sum_{i \in I_j} Y_j Y_i.
%$$
%Since $\E(Y_j^2) \leq 1/4$ we get
%$$
%\E( (X - \bar{p})^2) \leq 1/4 \cdot w + w \cdot L  \leq w(L+1).
%$$
%As $\bar{p} \geq p$, our claim follows.
%\end{proof}
%

\br

{\bf Step 4: Estimating the probability of success} 

\br

Fix $k_0 \in K_0 \cap \tilde K$ and fix a pure column based on $B_j$.
Since $B_j$ is the base of a pure column, either $\tau$ maps it to a $C$-level, or to a level not in $C$.
We will assume that the $i$-th level in the lower part of the column above $B_j$ is 
a $C$-level and that it is mapped by $\tau$ to a $C$-level.
The total $\mu_{k_0}$-measure of these $C$-levels is at least $\frac{1}{6} \mu_{k_0}(C)^2$ (see Step 2).

Suppose  
%$(y,v_1, v_2) \in  T_k^i B_j$ 
$(z, kH, kk_0, v_1, v_2) \in  T_{k_0}^i B_j$
and, it is mapped by $\tau$ to level $i' := N+ \pi(i)$.
Let
%$$
%\phi_{N-1 -i}(y, v_1) = h_1(y,v_1), \qquad \phi_{N-1 -i}(yk, v_2) = h_2(yk,v_2)
%$$
$$
h_1(z, kH, v_1) = \phi^\om_{N-1 -i}(z, kH, v_1) , \qquad  h_2(z, kk_0H, v_2) = \phi^\om_{N-1 -i}(z, kk_0H, v_2)
$$
and
%\begin{gather*}
%\phi_{i' - N-1}(T^{N - i +1}y, S_y^{(N-i)} v_1) = h'_1(yk,v_1), \\
%\phi_{i' - N-1}(T^{N - i +1} yk,S_{yk}^{(N-i)} v_2) = h'_2(yk,v_2).
%\end{gather*}
\begin{gather*}
h'_1(z, kH,v_1) = \phi^\om_{i' - N-1}(T^{N - i +1}(z,kH), S_{(z,kH)}^{(N-i)} v_1), \\
h'_2(z, kk_0H,v_2) = \phi^\om_{i' - N-1}(T^{N - i +1} (z,kk_0H),S_{(z,kk_0H)}^{(N-i)} v_2).
\end{gather*}

Suppose $T^{N -i} (z, kH, v_1) \in P \times Q_1 \times R_1$ and
$T^{N -i}(z, kk_0H, v_2) \in P \times Q_2 \times R_2$.
Then
%\begin{gather*}
%\psi^\om_\tau(y, v_1, v_2)  = 
%(\phi^\om_\tau(y,v_1), \phi^\om_\tau(yk, v_2)) = \\
%(h'_1(y,v_1) \cdot \xi_{P \times Q_1 \times R_1}^\om \cdot h_1(y,v_1),
%h'_2(y,v_2) \cdot \xi_{P \times Q_2 \times R_2}^\om \cdot  h_2(yk, v_2)).
%\end{gather*}
\begin{gather*}
\psi^\om_\tau(z, kH,  kk_0H, v_1,v_2)  = 
(\phi^\om_\tau(z, kH, v_1), \phi^\om_\tau(z, kk_0H, v_2)) = \\
(h'_1(z, kH,v_1) \cdot \xi_{P \times Q_1 \times R_1}^\om \cdot h_1(z, kH, v_1),
h'_2(z, kk_0H,v_2) \cdot \xi_{P \times Q_2 \times R_2}^\om \cdot  h_2(z, kk_0H, v_2)).
\end{gather*}

\br

The set $T_{k_0}^i B_j$ is partitioned into sets according to the atoms $P, Q_1, Q_2, R_1, R_2$, and
we denote an element of this partition by $\al$.
We let $X^\om_\al  =1$ iff $\psi^\om_\tau$ on $\al$ satisfies 
$$
d(\phi^\om_\tau(z, kH, v_1), g_1) < a \quad \& \quad d(\phi^\om_\tau(z, kk_0H, v_2), g_2) < a.
$$
Now, for $Q_1 \not = Q_2$ the corresponding random variables are independent 
and when we define pieces $\al$ and the corresponding $X_\al$ as before, 
we get $\Prob(X_\al =1) \geq \frac{1}{m^2}$.

\br

%As in Step 6 in the proof of Proposition \ref{open-dense} 

%We now want to apply Lemma \ref{simple}.

We will  need the following lemma, whose proof again is relegated to the appendix.

\begin{lem}\label{simple}
Let $p$ be a number in $(0,1)$ and let $L$ be a positive constant.
Suppose that $w_j \geq 0$ and $\sum_{j =1}^n w_j =1$.
Let $X_j, \ j=1,\dots,n$ be random variables taking values in $\{0,1\}$ with $P(X_j =1) \geq p$.
Suppose further that for each $j$ there is a set $I_j$, \  $ |I_j|  \leq L$,
such that for $i \not\in I_j$, $X_i$ and $X_j$ are independent.
Set $w = \max_{1 \leq j \leq n} w_j$.
Then for $X = \sum_{j=1}^n w_j X_j$ we have
$$
P(X \geq p/2)  \geq 1 - \frac{4(L+1)}{p^2} w.
$$
\end{lem}

\br

%To do this 
We need to estimate, for a fixed $X_\al$,  the number of $X_\beta$ that are not independent of it. 
Clearly if $\al$ and $\beta$ correspond to distinct $\Pcal_{l'}$ atoms, 
$X_\al$ and $X_\beta$ are independent. Therefore the number of ``bad" $\beta$'s is at most
$L : = |\Qcal_l| \cdot |\Rcal_l|^2$. By choosing $l'$ in (\ref{l'})
sufficiently large we can get the maximum weight $w$ in Lemma \ref{simple} small enough so that we get 
for all pure columns and their $C$-levels, for each fixed $k_0$, a set in $\Om$ of probability
at least $1 - b/10$, for which the $\mu_{k_0}$-measure of the set 
$(z, kH,  kk_0H, v_1,v_2) \in C$ in equation (\ref{ess1}) 
%set of $(z, kk_0H,v_2) , g_2)$ in equation (\ref{ess1}) 
is at least
$\frac{1}{100} \cdot \frac{1}{|F|^2} \cdot \mu_{k_0}(C)^2$.

We need to choose a single $\phi^\om$ which is good for a large set of $k_0 \in K_0 \cap \tilde K$.
To do this, as before we use Fubini's theorem on
$(\Om \times K, \Prob  \times \la_K)$.
The set of pairs $(\om, k_0)$ such that equation (\ref{ess1}) (in the definition of the set $U$) 
will hold for $\phi^\om$ by the above discussion, is at least $(1 - b/10)\la_K(K_0 \cap \tilde K)$.
Thus by Fubini's theorem there will certainly be a choice of $\om_0$, 
such that the set of $k_0$ for which (\ref{ess1}) will hold, has $\la_k$ measure at least $1 -b$, as required.
%
%
%
%
%
%As before, using Fubini's theorem on the probability space $(\Om \times K, \Prob \times \la_K)$,
%we find a subset of $K$ with $\la_K$ measure at least $1 - b$ for which equation (\ref{ess1}) holds.
The constant $c_a$ is now determined to be $\frac{1}{100} \cdot \frac{1}{|F|^2}$.

This concludes the proof of Proposition \ref{open-dense2}.
\end{proof}

%--------------------------------------------------

\br

Now back to the proof of Theorem \ref{max2}.
By Lemma \ref{BF} what we have to show is that the system
$\Xb_\phi \underset{\Zb} { \times} \Xb_\phi$ is relatively ergodic over 
$\Yb \underset{\Zb} { \times} \Yb$. 
As above we define the set $U= U(C,a, b, c_a, g_1, g_2) \subset \mathcal{C}$ 
and by Proposition \ref{open-dense2} we know that it is open and dense in $\Ccal$.

Now the ergodic components of $\Yb \underset{\Zb} { \times} \Yb$ are of the form
$$
Y_{k_0} = \{(z, kH, kk_0H) : z \in Z,\ k \in K\},
$$
and, as we have seen above, the corresponding ergodic component of $\Xb \underset{\Zb}{\times} \Xb$
has the form (\ref{Tk0})
\begin{equation*}
T_{k_0}(z, kH,  kk_0H, v_1, v_2) =
(Tz, \ga(z) kH, \ga(z) kk_0H, S_{(z,kH)} v_1, S_{(z, kk_0H)}v_2),
\end{equation*}
%$T_{k_0} : Z \times  \times V \times V \to Y \times V \times V$, where
%under $T_{k_0} \times T_\phi$ we have :
%$$
%T_{k_0}(y,  v_1, v_2) \mapsto (Ty, S_y v_1, S_{yk_0} v_2).
%$$
Thus for a fixed $k_0$ we have a cocycle 
$\psi_{k_0} : Z \times K/H \times K/H \times V \times V \to G \times G$, 
$$
\psi_{k_0}(z, kH, kk_0H, v_1, v_2)) =(\phi(z, kH, v_1), \phi(z, kk_0H, v_2)),
$$
which we must show is ergodic.

Fix $a$ and the corresponding $c_a$ and consider the union 
$K_\infty = \bigcup_{n}K_n \subset K$, with 
\begin{align*}
K_n  = \{k_0 \in K  &: \mu_{k_0}
(\{(z, kH, kH,  v_1, v_2)  \in C   : d(\phi_{\tau}(z, kH, v_1)) , g_1) < a \\
\ & \   \&  \  d(\phi_{\tau}(z, kk_0H, v_2), g_2) < a \}) 
> c _a \mu_{k_0}(C)^2\},
\end{align*}
so that $\la_K(K_n) > 1 - \frac{1}{2^n}$.
Then $\la_K(K_\infty)=1$, and for every $k_0 \in K_\infty$
\begin{align*}
\mu_{k_0}(\{(z, kH, kk_0H,  v_1, v_2) \in C & : \exists \tau \in [T]_f \ {\text {with}} \ 
d(\phi_{\tau}(z, kH, v_1) , g_1) < a \\
&  \   \&  \  d(\phi_{\tau}(z, kk_0H, v_2), g_2) < a\} )> c_a\mu_{k_0}(C)^2.
\end{align*}

Now with 
$$
\Ucal(C,a, c_a,g_1, g_2) = \bigcap_{n} U(C,a, \frac{1}{2^n}, c_a, g_1, g_2),
$$
we have a set $K_\infty =K_\infty(C, a, c_a, g_1, g_2) \subset K$, with $\la_K(K_\infty)=1$
such that for $k \in K_\infty$
$\Ucal(C,a, c_a,g_1, g_2)$ is dense $G_\del$ in $\Ccal$ (we eliminated the parameter $b$).

Now take an intersection 
$$
\Ucal(a,c_a, g_1, g_2) = \bigcap_{n \in \N}  \Ucal(C_n,a, c_a,g_1, g_2)
$$
over a dense collection $\{C_n\}_{n \in \N}$ in the measure algebra 
%$(\Ycal, \mu)$ 
of $Z \times K/H \times K/H \times V \times V$
and take the corresponding intersection 
$$
K_\infty(a, c_a, g_1, g_2)= \bigcap_{n \in \N} K_\infty(C_n, a, c_a, g_1, g_2).
$$ 
Next  take
$$
\Ucal(g_1, g_2) = \bigcap_{n \in \N} \Ucal(\frac1n,c_{\frac1n}, g_1, g_2)
$$
and the corresponding intersection
$$
K_\infty(g_1,g_2) = \bigcap_{n \in \N} L(\frac1n, c_{\frac1n}, g_1, g_2).
$$
Finally let
$$
\Ucal = \bigcap \{\Ucal(g_1, g_2) (g_1, g_2)  \in G_0 \times G_0\},
$$
with $G_0 \subset G$ a countable dense subset of $G$,
and the corresponding
$$
L = \bigcap \{K_\infty(g_1, g_2) : (g_1, g_2)  \in G_0 \times G_0\}.
$$
The existence of the set $L$ demonstrates the fact that for almost all $k \in K$,
and every cocycle $\phi$ in the residual set $\Ucal$,
every pair $(g_1,g_2) \in G \times G$ is an essential value for the cocycle $\psi_k$. 
It now follows that the system 
$\Xb_\phi \underset{\Zb} { \times} \Xb_\phi$ is relatively ergodic over 
$\Yb \underset{\Zb} { \times} \Yb$
and our proof of Theorem \ref{max2} is complete.
%----------------------------------------------------
%
%
%To deduce the statement of Theorem \ref{max2} from Proposition 
%\ref{open-dense2} we now follow the same procedure that we used
%in deducing Theorem \ref{max} from Proposition \ref{open-dense}.
%
%
%\br
%
%This concludes the proof of Theorem \ref{max2}.
%
\end{proof}

\br

The next theorem is actually a special case of Theorem \ref{max2} (when $\Xb = \Yb$),
however it is convenient for us to formulate it as a separate theorem,
in order to use it in the following proof of Theorem \ref{rank-succ}.

\begin{thm}\label{max}
Let  $\Yb \to \Zb$ be an infinite factor of the ergodic system $\Yb$
such that the extension $\Yb \to \Zb$ is a $K/H$-extension.
% and that it is the maximal compact extension of $\Zb$ in $\Xb$. 
Then for a generic $\phi \in \mathcal{C}(Y,G)$ the system 
$\Yb_\phi = (Y \times G, \mathcal{Y} \times \Bcal_G, \mu \times \la_G,T_\phi)$,
is ergodic and the projection map, $\pi : \Yb \to \Zb$, 
is the maximal compact extension of $\Zb$ within $\Yb_\phi$.
\end{thm}

\br

We can now complete the proofs of Theorems \ref{rank-succ} and \ref{thm-main}.

\br

\begin{proof}[Proof of Theorem \ref{rank-succ}]
Apply a transfinite induction along the ordinal $\al$, using  Theorem \ref{max} 
for  successor ordinals and,  Lemma \ref{IL} and Theorem \ref{max2} for limit ordinals.
\end{proof}

\br
%
%
%\br
%
%\section{The proof of Theorem \ref{tower}}\label{sec-proof}
%
%\br

%We are are now able to complete the proof of Theorem \ref{tower} :

\begin{proof}[Proof of Theorem \ref{thm-main}]
%In view of the definition of the distal rank of an ergodic system and Lemma \ref{max},
%all we have to do is prove the theorem by transfinite induction on the ordinals $\{ \al < \om_1\}$, 
%when for the step $\al \to \al +1$ we apply Theorem \ref{rank-succ}.
As was explained in Section \ref{sec-rank} the combination of Theorems 
\ref{rank-succ} and Theorem \ref{rank-limit}
implies Theorem \ref{thm-main}.
\end{proof}

\br

\section{Some corollaries}\label{sec-cor}

\br

\begin{setup}\label{setup-cor}
Let $K$ and $G$ be two arbitrary second countable compact groups.
 Suppose that $\Zb = (Z, \mathcal{Z}, \nu, R)$ is an ergodic $\Z$-system and 
 suppose that $\Yb = (Y, \mathcal{Y}, \mu,T)$, with
 $\mu = \nu \times \la_K$, $\la_K$ being the Haar measure on $K$,
 is a $K$-extension of $\Zb$.
 %, so that $Y = Z \times K$ and 
% there is a measurable map $\kappa : Z \to K$ such that
% $T(z,k) = (Rz, \kappa(z)k) \  (z \in Z, k \in K$).
 Set $X = Y \times G = Z \times K \times G$ and let $\mathcal{C} = \Ccal(Y,G)$ 
 be the Polish space of Borel maps $\phi : Y \to G$.
 For $\phi \in \mathcal{C}$ set 
 $$
 T_\phi(y, g) = (Ty, \phi(y)g), \quad (y \in Y, g \in G).
 $$
\end{setup} 

\br

\begin{thm}\label{fixedk}
For any $n \geq 1$ and fixed $(k_1, k_2, \dots, k_n)$ of distinct elements of $K$ 
there is a dense $G_\del$ subset $\Ccal_0 \subset \Ccal(Y,G)$ such that:
\begin{enumerate}
\item
For $\phi \in \Ccal_0$ the corresponding cocycle
$$
\phi_{(k_1, k_2, \dots, k_n)} : Y \to G^{n+1}
$$
defined by
$$
\phi_{(k_1, k_2, \dots, k_n)}(y) = (\phi(y), \phi(yk_1), \dots, \phi(yk_n))
$$
is ergodic.
\item
Furthermore, denoting by $\Yb_{(k_1, k_2, \dots, k_n)}$ the corresponding 
ergodic skew product on $Y \times G^{n+1}$, we have that 
the extension $\Yb \to \Zb$ is the largest compact extension of 
$\Zb$ in $\Yb_{(k_1, k_2, \dots, k_n)}$.
\item
Denoting, for $k \in K$, by $\phi_k : Y \to G$ the cocycle $\phi_k(y) = \phi(yk)$, we have that the corresponding
skew products $\Yb_{\phi_{k_i}}$, $i=0,1,\dots,n$, are jointly disjoint over their common factor $\Yb$.
\end{enumerate}
\end{thm}

\begin{proof}
(1) and (2) :
The proof is similar to the proof of Theorem \ref{max}.
The situation here differs in two ways.
The first is the fact that we are considering here the groups $K^n$ and $G^{n+1}$
rather than $K$ and $G$. This change however does not cause any new difficulty,
the generalization is straightforward.
The second way actually makes the proof much easier since
we now have to deal with a single element of
the group $K^n$, rather than a subset of $K^n$ with large Haar measure.
Following the proof of Theorem \ref{max},  one formulates a proposition similar to
Proposition \ref{open-dense2}, where now the definition of the open set $U$
is simplified, as follows :

Fix a subset $C \subset Y, \ \mu(C) >0$, elements $g_1, \dots, g_{n+1}$ in $G$, a positive small
constant $a$, and a positive constant $c_a$ (which will depend on $a$). 
Define the set $U(C,a, c_a, g_1, g_2, \dots, g_{n+1}) \subset \mathcal{C}$ 
as the collection of all the functions $\phi \in \mathcal{C}$
with the following property :

\br 

There exists an element $\tau \in [T]_f$ such that :
\begin{enumerate}
\item
$ \tau (C) = C$,
 \item
 With $k_0 =e$,
$$
\mu(\{y \in C :  \forall j, \ 0 \leq j \leq n,
\  d(\phi_{\tau}(yk_j), g_{j+1}) < a \}) > c _a \mu(C).
$$
\end{enumerate}
We leave the details of the proof to the reader.

(3) Since the skew product system $\Yb_{(k_1, k_2, \dots, k_n)}$ is clearly isomorphic to
the relative independent product of $\prod_{i=0}^n \underset{\Yb}{ }  \Yb_{\phi_i}$,
it follows that the latter's relative product measure is ergodic and thus, by \cite[Theorem 3.30]{G-03},
it is the unique invariant measure on $\Yb_{(k_1, k_2, \dots, k_n)}$ which projects onto $\mu$.
\end{proof}

\br

\begin{thm}\label{Mych}
\begin{enumerate}
\item
There is a dense $G_\del$ subset $\Ccal_1 \subset  \Ccal(Y,G)$ 
such that for each $\phi \in \Ccal_1$ and every $n \geq 1$,
there is a dense $G_\del$ subset $A_n \subset K^n$
with the property that for every $(k_1, k_2, \dots, k_n) \in A_n$
the corresponding cocycle $\phi_{(k_1, k_2, \dots, k_n)}$
from $Y$ to $G^{n+1}$ is ergodic.
\item
There is a Cantor subset $K_0 \subset K$ such that for every
$n \ge 1$ and every $(k_1, k_2, \dots, k_n)$ with $k_i \in K_0, \ i=1,2,\dots,n$,
the corresponding cocycle $\phi_{(k_1, k_2, \dots, k_n)}$
from $Y$ to $G^{n+1}$ is ergodic, so that, as in Theorem \ref{fixedk}(3), the skew products
$\Yb_{\phi_{k_i}}$, $i=0,1,\dots,n$, are jointly disjoint over their common factor $\Yb$. 
\end{enumerate}
\end{thm}

\begin{proof}
The first claim follows from the Kuratowski--Ulam theorem (see, for example, \cite[Theorem 8.41]{K}).
Then the second claim follows by Mycielski's theorem (see, for example, \cite[Theorem 19.1]{K}).
\end{proof}

\br

\begin{thm}\label{JPaug}
Let $K$ be an infinite monothetic compact second countable topological group with a topological generator $a$
(i.e. $K = \ol{\{a^n : n \in \Z\}}$). Let $\Kb =(K, \Bcal_K, \la_K, T_a)$ be the corresponding
ergodic system.
Let $G$ be an arbitrary compact second countable topological group.
Then for a generic cocycle $\phi \in \Ccal(K,G)$ the corresponding skew product 
system $\Xb_\phi = (K \times G, \Bcal_K \times \Bcal_G, \la_K \times \la_G, T_\phi)$
is (i) ergodic and (ii) the system $\Kb$ is the largest compact factor of $\Xb_\phi$.
In particular then $\Xb_\phi$ is distal of rank $2$.
\end{thm}

\begin{proof}
This is in fact a special case of Theorem \ref{max}.
\end{proof}

\br

\section{The case of a weakly mixing extension}\label{sec-wm-ext}

\begin{setup}\label{setup-wm}
Let  $\Yb  \to \Zb$ be a nontrivial relatively weakly mixing extension of ergodic systems. 
Let $G$ be a second countable compact topological group with Haar measure $\la_G$.
Let $\Ccal= \Ccal(Y,G)$ be the Polish space of measurable cocycles $\phi : Y \to G$.
For $\phi \in \Ccal$ we let 
$$
T_\phi (y, g) = (Ty, \phi(x)g), \ y \in Y, g \in G,
$$
and $\Yb_\phi = (Y \times G, \Ycal \times \Bcal_G, \mu \times \la_G, T_\phi)$.
We let $\psi : Y \times Y \to G \times G$ be the cocycle
$\psi(y_1, y_2) = (\phi(y_1), \phi(y_2))$.
\end{setup}

\br

\begin{thm}\label{rwm}
For a generic cocycle $\phi \in \Ccal(Y,G)$ and the corresponding skew product transformation 
$T_\phi : Y \times G \to Y \times G$, the extension $\Yb_\phi \to \Zb$ is relatively weakly mixing.
In particular, when $\Yb$ is weakly mixing (i.e. when $\Zb$ is the trivial one point system)
so is the system $\Yb_\phi$ for a generic cocycle $\phi \in \Ccal(Y,G)$. 
\end{thm}

\begin{proof}
Let $\pi : Y \to Z$ be the factor map $\Yb \to \Zb$.
Our assumption is that the system $\Wb = \Yb \underset{\Zb}{\times} \Yb$
with 
$$
W = Y \underset{Z}{\times} Y =  \{(y_1, y_2) : \pi(y_1) = \pi(y_2)\} \subset Y \times Y,
$$
$T(y_1, y_2) = (Ty_1, Ty_2)$, and $\zeta = \mu \underset{\pi}{\times} \mu$
the relative product measure over $\pi$, is ergodic.

What we have to show is that  for a generic $\phi \in \Ccal$ the system
$$
\Wb_\psi  = (W_\psi, \tilde{\zeta}, T_\psi),
$$
where
$$
W_\psi =  W \times G \times G
= \{((y_1,g_1), (y_2, g_2) : \pi(y_1) = \pi(y_2)\}, 
$$
$\tilde{\zeta}  = \zeta \times \la_G \times \la_G$ 
and $T_\psi(y_1, y_2, g_1, g_2) = (Ty_1,T y_2, \phi(y_1)g_1, \phi(y_2) g_2)$, is ergodic.

Define, for $C \subset W =  Y\underset{Z}{\times} Y$, 
$a > 0$, $c_a$ a constant depending on $a$, and elements 
$g_1, g_2 \in G$, a subset $U  = U(C, a, c_a, g_1, g_2)
\subset \Ccal(Y,G)$ as follows
\begin{align*}
U  =
\big{\{}\phi \in \Ccal & : \exists \ \tau \in [T \times T]_f \ {\text{with}}\\
& \ \ \  \zeta  ( \{(y_1,y_2) \in C : \tau(y_1, y_2) \in C,\\
&  \ \ \  {\text{and}}\ 
d(\psi_\tau(y_1, y_2), (g_1,g_2)) < a\})  \\
&  \ \ \ > c_a(\mu \times \mu)(C) \big{\}}.
\end{align*}
Here $\tau(y_1, y_2) = (T^{\sig(y_1,y_2)} y_1, T^{\sig(y_1,y_2)} y_2)$, 
for a function $\sig : W \to \Z$.

\begin{prop}\label{open-dense3}
For a sufficiently small $c_a$ the set $U(C,a, c_a, g_1, g_2)$ is open and dense in $\mathcal{C}(Y, G)$.
\end{prop}

\begin{proof}
It is clear that $U$ is an open set. We will show that it is dense in $\Ccal$.
So fix $\phi_0 \in \Ccal$ and $\del >0$, with $\del \ll \zeta(C)$.
Take a set $A \subset X$ with $\mu(A) < \del/10$.
Let $B \subset W$ be a base of a Rokhlin tower of height $3N$, such that 
(using the ergodic theorem for $T \times T$ on $\Wb$), when the tower is purified with respect to
$A \underset{Z}{\times} A$ and $C$, the
following holds : the pure columns that satisfy
\begin{enumerate}
\item[(i)]
the number of $C$-levels in the top and bottom thirds of the column are at least 
$\frac{1}{2} \zeta(C) \cdot N$,
\item[(ii)]
in the middle third there is at least one level in $A \underset{Z}{\times} A$,
\end{enumerate}
fill $(1 - \frac{\del}{10})$ of $W$.

\br

Now define $\tau$ on good pure columns by exchanging $C$-levels in the lower and upper thirds,
and identity elsewhere.

\br

By \cite{W} we can assume that in the diagram $\pi : \Yb \to \Zb$,
the ergodic systems $\Yb$ and $\Zb$ are represented as topological strictly ergodic flows;
i.e. $Y$ and $Z$ are metric compact spaces, $T$ acts on both as a homeomorphism 
under which the topological flows are minimal and uniquely ergodic, and finally that $\pi$
is a continuous map. As we assume that the extension $\pi$ is nontrivial and 
relatively weakly mixing, it follows that the minimal flow $(Y,T)$ has no finite orbits.

Working with these models we see that the in the system $\Wb$ the space
$W =  Y\underset{Z}{\times} Y$ is compact metric and the transformation $T \times T : W \to W$
is a homeomorphism.

\br

%We may assume that $X$ is a compact metric space, that $T$ is a homeomorphism of $X$,
%and that the topological dynamical system $(X,T)$ is infinite and minimal
%(so in particular there are no finite orbits).
Given $\eta >0$ set
\begin{align*}
E   = \{ (y_1, y_2) \in W : & \  \min_{0 \leq i < j \le 3N} d(T^i y_1, T^j y_1) > \eta, \quad
 \min_{0 \leq i , j \le 3N} d(T^i y_1, T^j y_2) > \eta,\\
& \  \min_{0 \leq i < j \le 3N} d(T^i y_1, T^j y_2) > \eta, \quad
\min_{0 \leq i < j \le 3N} d(T^i y_2, T^j y_2) > \eta  \}.
\end{align*}
It then follows that for a sufficiently small $\eta >0$ we will have
$\zeta(E) > 1 - \frac{\del}{100}$.
%\begin{align*}
%(\mu \times \mu)  \big{(}\{ (x_1, x_2) : & \  \min_{0 \leq i < j \le 3N} d(T^ix_1, T^jx_1) > \eta, \quad
% \min_{0 \leq i , j \le 3N} d(T^ix_1, T^jx_2) > \eta,\\
%& \  \min_{0 \leq i < j \le 3N} d(T^ix_1, T^jx_2) > \eta, \quad
%\min_{0 \leq i < j \le 3N} d(T^ix_2, T^jx_2) > \eta  \} \big{)}\\
%&  \ > 1 - \frac{\del}{100}.
%\end{align*}
We call $E$ the {\em $\eta$-separated set}.

Next let $\Pcal$ be a finite partition of $A$ with 
$$
\max_{P \in \Pcal}  \diam(P) < \frac{\eta}{100}.
$$
Let $F = \{f_1, f_2, \dots, f_m\} \subset G$ be a set such that
$\bigcup_{i =1}^m B(f_i, \frac{a}{10})  = G$.
We define a random cocycle
$$
\phi^\om(y) =
\begin{cases}
\phi_0(y), & x \not\in A_0\\
\xi^\om_P, & y \in P  \in \Pcal,
\end{cases}
$$
where $\xi^\om_P$ are random variables on $\Om$ with
$$
\Prob(\xi^\om_P =f_i) = 1/m, \quad f_i \in F
$$ 
and independent for $P \not= P'$.
%uniform on $F$, and independent
%for distinct $P \in \Pcal$.
If $y \in A$ and $y \in P$ we write $P = P(y)$ so that $y \in P(y)$.

% 
%For $(y_1, y_2)$ with
%$(T \times T)^n(y_1, y_2) \in P \times Q \in \Pcal \times \Pcal$ we write 
%$P_n = P$ and $Q_n = Q$.

\br

Now, for $(y_1, y_2)$ in a $C$-level of a pure column 
$i, \ 0 \le i <N$, which is mapped by $\tau$ to the $C$-level 
$i', \ 2N < i' \le 3N$, the cocycle $\psi^\om_\tau$ takes the following form :
\begin{equation}\label{psi-tau}
\psi^\om_\tau(y_1, y_2) = 
\prod_{n =0}^{i' - i -1} \rho^\om_n(y_1, y_2),
\end{equation}
where the random variables $ \rho^\om_n$ are defined as follows
\begin{equation*}
\rho^\om_n(y_1, y_2) =
\begin{cases}
(\xi^\om_{P(T^n y_1)}, \xi^\om_{P(T^n y_2)}) , & {\text{if}} \ T^n y_1  \in A  \ \& \  T^n(y_2) \in A \\
(\phi_0(T^n y_1), \phi_0(T^n y_2)), & {\text{if}}\ T^n y_1 \not \in A  \ \& \  T^n(y_2) \not\in A \\
(\xi^\om_{P(T^n y_1)}, \phi_0(T^n y_2)), & {\text{if}}\ T^n y_1 \not \in A \ {\text{and}}\ T^n y_2 \in A\\
(\phi_0(T^n y_1), \xi^\om_{P(T^n y_2)}), & {\text{if}} \ T^n y_1 \not \in A \ {\text{and}} \ T^n y_2 \in A.
\end{cases}
\end{equation*}

\br

When $(y_1,y_2)$ is in the $\eta$-separated set, all 
$P(T^n y_1), P(T^n y_2)$ that appear in the product for $\psi^\om_\tau$ are distinct
and there is at least one $(P(T^n y_1), P(T^n y_2))$ in the product by our assumption
that in the middle third, that we have to pass through, there is at least one level in $A \times A$.

all the $P_n$ and $Q_n$ that appear are distinct,
and there is at least one $P_n \times Q_n$ in the product (\ref{psi-tau}) since
we pass through the middle third where there is at least one $(A \times A)$-level.

It follows that for each such $(y_1, y_2)$
$$
\Prob (d(\psi^\om_\tau(x_1, x_2), (g_1, g_2)) < a) \ge \frac{1}{m^2}.
$$
Applying Fubini's theorem to 
$\Om \times (C \cap (\bigcup_{i =0}^{N-1} (T \times  T)^i B))$ we see that there is a choice of $\om$
such that 
\begin{align*}
\zeta( \{(y_1,y_2) \in C & : \tau(y_1, y_2) \in C,
 \  {\text{and}}\ 
d(\psi^\om_\tau(y_1, y_2), (g_1,g_2) < a\} ) \\
&  \ \ \ >  \frac{1}{10 m^2} \cdot \zeta(C).
\end{align*}
We now let $c_a = \frac{1}{10 m^2}$ and observe that
this $\phi^\om$ is $\del$-close to $\phi_0$ and lies in $U$.
Thus we have shown that $U$ is an open and dense subset of $\Ccal$.
\end{proof}

\br

%\begin{proof}[Proof of Theorem \ref{rwm}]
To deduce the statement of Theorem \ref{rwm} we now follow the same procedure that we used
in deducing Theorem \ref{max} from Proposition \ref{open-dense2}.
%
%
%\br
%
%This concludes the proof of Theorem \ref{wm}.
%
\end{proof}

\br

We now have the following corollary which is a far reaching strengthening of Theorem \ref{generica}.

\begin{thm}\label{generical} 
Let $\Yb \to \Zb$ be a factor map of ergodic systems. 
Let  $\Yb \to \Yb_{rd} \to \Zb$ be the (relative) Furstenberg-Zimmer structure for the extension $\Yb \to \Zb$.
Then for a generic $\phi \in \Ccal(Y,G)$ we have that 
$\Yb_\phi \to \Yb_{rd} \to \Zb$ is the (relative) Furstenberg-Zimmer structure for the extension $\Yb_\phi \to \Zb$.
In particular, when we take $\Zb$ to be the trivial one point system, it follows that for every ergodic system $\Yb$
with Furstenberg-Zimmer structure $\Yb \to \Yb_{d}$ (the latter being the largest distal factor $\Yb$),
for a generic $\phi \in \Ccal(Y,G)$ the corresponding system $\Yb_\phi$ is ergodic with 
Furstenberg-Zimmer structure $\Yb_\phi \to\Yb_d$.
%Moreover, if $\Yb = \Yb_{rd}$; i.e. $\Yb \to \Zb$  is a distal extension, then
%the rank of the relative distal tower of $\Yb_\phi$ is that of $\Yb$ plus $1$. 
\end{thm}

\begin{proof}
%Consider the diagram $\Xb_\phi \to \Yb \to \Zb$.
Let $\Yb \to \Yb_{rd} \to \Zb$ be the Furstenberg-Zimmer structure for the extension $\Yb \to \Zb$;
i.e. $\Yb_{rd}$ is the largest relative distal extension of $\Zb$ in $\Yb$, and the extension $\Yb \to \Yb_{rd}$ is 
relatively weakly mixing. If this latter extension is non-trivial then, by Theorem  \ref{rwm},
for a generic $\phi \in \Ccal(Y,G)$, 
the extension $\Yb_\phi \to \Yb_{rd}$ is relatively weakly mixing.
On the other hand, if $\Yb = \Yb_{rd}$ then we use Theorem \ref{max2}.
%$\Yb_\phi \to \Zb$ is a relative distal extension.
\end{proof}

\br

We also have the following theorem:

\begin{thm}\label{generall}
Let $\Yb \to \Zb$ be an infinite factor of the ergodic system $\Yb$ 
and let $G$ be a compact second countable topological group.
Then for a generic $\phi \in \Ccal(Y,G)$, with 
$\Yb_\phi = (Y \times G, \Ycal \times \Bcal_G, \mu \times \la_G, T_\phi)$,
the extension 
$$
 \Yb_\phi \underset{\Zb} { \times} \Yb_\phi \to \Yb \underset{\Zb} { \times} \Yb
$$
is relatively ergodic.
%Equivalently, the system $\Yb_\phi$ is $2$-fold ergodic over $\Yb \to \Zb$.
\end{thm}

\begin{proof}
Let $\tilde{\Zb} \to \Zb$ be the maximal compact extension of $\Zb$ in $\Yb$.
If this factor map is trivial, that is if the extension $\Yb \to \Zb$ is relatively weakly mixing,
the assertion holds by Theorem \ref{rwm}. Otherwise, by the relative version
of the Furstenberg-Zimmer theorem, the extention $\tilde{\Zb} \to \Zb$
is a nontrivial  $K/H$-extension and the assertion follows from Theorem \ref{max2}.
\end{proof}

\br

\section{A general framework and a master theorem}\label{sec-con}

%
%\begin{lem}\cite[Lemma 2.8]{BF}\label{BF}
%Let $\Zb$ be a factor of $\Xb$ and $\Yb$ be a compact extension of $\Zb$  in $\Xb$
%(i.e. $\Yb$ is also a factor of $\Xb$). Then $\Yb$ is the maximal compact extension of
%$\Zb$ in $\Xb$ iff $\Xb \underset{\Zb} { \times} \Xb$ is relatively ergodic over 
%$\Yb \underset{\Zb} { \times} \Yb$.
%\end{lem}

\begin{defn}
Given a diagram $\Xb \to \Yb \to \Zb$ with $\Xb$ an ergodic system 
% and a compact second countable topological group $G$, 
we say that {\em $\Xb$ is  $2$-fold ergodic over} $\Yb \to \Zb$ when
%We say that
%an extension $\Wb \to \Xb$, with $\Wb$ is ergodic, is {\em $2$-fold ergodic over
%$\Yb \to \Zb$} if
the extension
\begin{equation}\label{rel-erg}
 \Xb \underset{\Zb} { \times} \Xb  \to \Yb \underset{\Zb} { \times} \Yb
\end{equation}
%$$
% \Wb \underset{\Zb} { \times} \Wb \to \Yb \underset{\Zb} { \times} \Yb
%$$
is relatively ergodic.

Given a diagram $\Wb \to \Xb \to \Yb \to \Zb$ with $\Wb$ an ergodic system
such that the condition (\ref{rel-erg}) is satisfied, we will say that 
{\em the extension $\Wb \to \Xb$ is $2$-fold ergodic over $\Yb \to \Zb$}.

%When $\Wb = \Xb$, so that our condition now is that the extension 
%$$
% \Xb \underset{\Zb} { \times} \Xb  \to \Yb \underset{\Zb} { \times} \Yb
%$$
%is ergodic,
%we say that $\Xb$ is {\em $2$-fold ergodic over $\Yb \to \Zb$}.

Let $G$ be a compact second countable topological group.
Given a diagram $\Xb \to \Yb \to \Zb$ with $\Xb$ an ergodic system,
{\em a cocycle $\phi \in \Ccal(X,G)$ is $2$-fold ergodic over
$\Yb \to \Zb$} if the corresponding extension $\Xb_\phi \to \Xb$
is {\em $2$-fold ergodic over} $\Yb \to \Zb$; i.e. when the extension
$$
 \Xb_\phi \underset{\Zb} { \times} \Xb_\phi \to \Yb \underset{\Zb} { \times} \Yb
$$
is relatively ergodic.
% $\Xb \underset{\Zb} { \times} \Xb$ is relatively ergodic over 
%$\Yb \underset{\Zb} { \times} \Yb$.
%We say that a cocycle $\phi \in \Ccal(Y,G)$ is {\em $2$-fold ergodic over $\Zb$} 
%when the corresponding extension $\Xb_\phi$ is $2$-fold ergodic over $\Yb \to \Zb$. 
This is the same as saying that for the cocycle $\psi : X \to G \times G$, given by
$\psi(x_1, x_2) = (\phi(x_1), \phi(x_2))$,
the extension 
$$
(\Xb \underset{\Zb}{\times} \Xb)_\psi \to \Yb \underset{\Zb}{\times} \Yb
$$
is relatively ergodic.
Using this terminology the assertion of Lemma \ref{BF} is that when 
$\Xb \to \Yb \to \Zb$ is a diagram of ergodic systems and
$\Yb \to \Zb$ is a compact extension, then
$\Yb$ is the maximal compact extension of
$\Zb$ in $\Xb$ iff the system $\Xb$ is {\em $2$-fold ergodic over} $\Yb \to \Zb$.

\br

More generally, given a diagram $\Wb \to\Xb \to \Yb \to \Zb$ with $\Wb$ an ergodic system,
we define, for every $n \ge 2$, the notions of {\em $n$-fold ergodicity} of the extension 
$\Wb \to Xb$ over $\Yb \to \Zb$, and of $\phi \in \Ccal(X,G)$ over $\Yb \to \Zb$. 
E.g. {\em $\Xb$ is $3$-fold ergodic over $\Yb \to \Zb$} when the extension
$$
\underset{\Zb}{\Xb \times \Xb \times \Xb} \to \underset{\Zb}{\Yb \times \Yb \times \Yb}
$$
is relatively ergodic, 
and {\em the cocycle $\phi \in \Ccal(X,G)$ is $3$-fold ergodic over $\Yb \to \Zb$}
when the cocycle $\psi(x_1, x_2, x_3) = (\phi(x_1), \phi(x_2), \phi(x_3))$ is relatively ergodic over 
$\Yb \underset{\Zb}{\times} \Yb$. 
\end{defn}
\noindent It is not hard to see that $2$-fold ergodicity implies $n$-fols ergodicity for all $n \ge 3$.

\br

%
%\begin{conj}\label{general}
%Let $G$ be a compact second countable topological group with normalized Haar measure $\la_G$.
%Let $\Xb \to \Yb \to \Zb$ be a chain of factors of the ergodic system $\Xb$ with $\Zb$ infinite. 
%Suppose further that the extension 
%$$
% \Xb \underset{\Zb} { \times} \Xb \to \Yb \underset{\Zb} { \times} \Yb
%$$
%is relatively ergodic.
%Then for a generic $\phi \in \Ccal(X,G)$, with 
%$\Xb_\phi = (X \times G, \Xcal \times \Bcal_G, \mu \times \la_G, T_\phi)$,
%the extension 
%$$
% \Xb_\phi \underset{\Zb} { \times} \Xb_\phi \to \Yb \underset{\Zb} { \times} \Yb
%$$
%is relatively ergodic.
%Equivalently, 
%% with respect to the diagram $\Xb \to \Yb \to \Zb$, 
%the extension $\Xb_\phi \to \Xb$ is $2$-fold ergodic over $\Yb \to \Zb$.
%\end{conj}

%---------------

\begin{lem}\label{re>rwm}
Let $\Xb \to \Yb \to \Zb$ with $\Xb$ ergodic be given.
If the extension $\Xb \underset{\Zb} { \times} \Xb \to \Yb \underset{\Zb} { \times} \Yb$ is
relatively ergodic then the extension $\Xb \to \Yb$ is relatively weakly mixing.
\end{lem}

\begin{proof}
Suppose $\Xb \to \Yb$ is not relatively weakly mixing. Then there exists
an intermediate factor $\Xb \to \Eb \to \Yb$ with $\Eb \to \Yb$ a nontrivial compact extension.
It follows that the extension $\Eb \underset{\Zb}{\times} \Eb \to \Yb  \underset{\Zb}{\times} \Yb$
is not relatively ergodic and a fortiori also
$\Xb \underset{\Zb} { \times} \Xb \to \Yb \underset{\Zb} { \times} \Yb$ is not relatively ergodic,
contradicting our assumption.
\end{proof}

\begin{lem}\label{rwm+rwm}
Let $\Xb \to \Yb \to \Zb$ with $\Xb$ ergodic be given.
If the extensions $\Xb \to \Yb$ and $\Yb \to \Zb$ are relatively weakly mixing
then so is the extension $\Xb \to \Zb$. 
\end{lem}

\begin{proof}
Suppose $\Xb \to \Zb$ is not relatively weakly mixing. Then there exists 
an intermediate factor $\Xb \to \Eb \to \Zb$ with $\Eb \to \Zb$ a nontrivial compact extension.
By our assumptions
$\Eb \to \Zb$ is compact and $\Yb \to \Zb$ is relatively weakly mixing.
It then follows that  $\Yb$ and $\Eb$ are relatively disjoint over their common factor $\Zb$,
and we have $\Xb \to \Yb  \underset{\Zb} { \times} \Eb \to \Yb$.
(For this we refer e.g. to \cite[Theorem 6.27]{G-03}, where this claim is proven in the absolute case;
however the same proof works also in the relative case.)
Now clearly $(\Yb  \underset{\Zb} { \times} \Eb)\underset{\Yb} { \times} (\Yb  \underset{\Zb} { \times} \Eb)$
is not ergodic, but we also have
$\Xb \underset{\Yb} { \times} \Xb \to 
(\Yb  \underset{\Zb} { \times} \Eb)\underset{\Yb} { \times} (\Yb  \underset{\Zb} { \times} \Eb)$
and we arrived at a contradiction.
\end{proof}

%-----------------

Combining the $2$-fold ergodicity theorems proven so far, we can now state and prove 
a master theorem on  $2$-fold ergodicity.

\begin{thm}\label{general}
Let $G$ be a compact second countable topological group with normalized Haar measure $\la_G$.
Let $\Xb \to \Yb \to \Zb$ be a chain of factors of the ergodic system $\Xb$ with $\Zb$ infinite. 
Suppose further that the extension 
$$
 \Xb \underset{\Zb} { \times} \Xb \to \Yb \underset{\Zb} { \times} \Yb
$$
is relatively ergodic.
Then for a generic $\phi \in \Ccal(X,G)$, with 
$\Xb_\phi = (X \times G, \Xcal \times \Bcal_G, \mu \times \la_G, T_\phi)$,
the extension 
$$
 \Xb_\phi \underset{\Zb} { \times} \Xb_\phi \to \Yb \underset{\Zb} { \times} \Yb
$$
is relatively ergodic.
Equivalently, 
% with respect to the diagram $\Xb \to \Yb \to \Zb$, 
the extension $\Xb_\phi \to \Xb$ is $2$-fold ergodic over $\Yb \to \Zb$.
\end{thm}

\begin{proof}
Let $\Yb \to \Yb_{rd} \to \Zb$ be the Furstenberg-Zimmer structure for the extension $\Yb \to \Zb$.
Then the fact that the extension $\Yb \to \Yb_{rd}$ is relatively weakly mixing, combined with 
our assumption imply that also the extension $\Xb \to \Yb_{rd}$ is relatively weakly mixing;
that is,  $ \Xb_{rd} = \Yb_{rd}$ and 
$\Xb \to   \Yb_{rd} \to \Zb$ is the Furstenberg-Zimmer structure for the extension $\Xb \to \Zb$.

Indeed, by Lemma \ref{re>rwm} the extension $\Xb \to \Yb$ is relatively weakly mixing.
Also, by definition the extension $\Yb \to \Yb_{rd}$ is relatively weakly mixing.
Hence, by Lemma \ref{rwm+rwm} the extension $\Xb \to \Yb_{rd}$ is relatively weakly mixing.

%
%------------
%
%
%
%We first show that the extension $ \Yb \underset{\Zb} { \times} \Yb \to 
%\Yb_{rd} \underset{\Zb} { \times} \Yb_{rd}$ is relatively ergodic. 
%To see this let $f$ be an invariant function in 
%$L_2(\Yb \underset{\Zb} { \times} \Yb)$, and let $\tilde{\Yb}$ be the maximal compact 
%extension of $\Zb$ in $\Yb$.
%Then, by \cite[Theorem 7.1]{F-77}, we conclude that 
%$f \in L_2(\tilde{\Yb} \underset{\Zb} { \times}  \tilde{\Yb})$, whence also in
% $L_2(\Yb_{rd} \underset{\Zb} { \times}  \Yb_{rd})$.
%  
%
%
%
%
%
%
%
%
%
%To see this let $f$ be an invariant function in $L_2(\Xb \underset{\Zb} { \times} \Xb)$.
%Applying \cite[Theorem 7.1]{F-77}, we conclude that 
%
%
%Theorem \ref{generical} implies that 
%for a generic $\phi \in \Ccal(X,G)$, the diagram
%$$
%\Xb_\phi \to \Yb_{rd} \to \Zb
%$$
%is the Furstenberg-Zimmer relative tower of $\Xb_\phi$ over $\Zb$.
%
%-----

\br

 {\bf Case 1: } 
 Suppose first that  $\Yb_{rd} \to \Zb$ is nontrivial; then there is a nontrivial intermediate extension
 $\Xb_\phi \to  \Zb_1 \to \Zb$, so that $\Zb_1$ is the maximal compact extension of $\Zb$ in $\Xb_\phi$.
By Lemma \ref{BF} the extension
$$
 \Xb_\phi \underset{\Zb} { \times} \Xb_\phi \to \Zb_1 \underset{\Zb} { \times} \Zb_1
$$
is relatively ergodic and, since,
for a generic $\phi$, by Theorem \ref{generical}, $\Zb_1$ is also a factor of $\Yb$, a fortiori also
$$
 \Xb_\phi \underset{\Zb} { \times} \Xb_\phi \to \Yb \underset{\Zb} { \times} \Yb
$$
is relatively ergodic.
(In fact, if $f \in L_2(\Xb_\phi \underset{\Zb} { \times} \Xb_\phi)$ is $T$-invariant then, by relative ergodicity,
it is $\Zb_1 \underset{\Zb} { \times} \Zb_1$-measurable,
and, a fortiori, also $\Yb \underset{\Zb} { \times} \Yb$.)

 \br
 
 {\bf Case 2: } 
In the remaining case the extension $\Xb_\phi \to \Zb$ is relatively weakly mixing; i.e.
the system $\Xb_\phi \underset{\Zb} { \times} \Xb_\phi$ is ergodic,
and clearly then the extension 
$$
 \Xb_\phi \underset{\Zb} { \times} \Xb_\phi \to \Yb \underset{\Zb} { \times} \Yb
$$
is relatively ergodic.
%Our assumption then implies that the extension $\Xb \to \Zb$ is
%relatively weakly mixing and our claim follows from Theorem \ref{rwm}
%with $\Xb_\phi = (X \times G, \Xcal \times \Bcal_G, \mu \times \la_G, T_\phi)$,
%In particular the extension 
%$$
% \Xb_\phi \underset{\Zb} { \times} \Xb_\phi \to \Yb_{rd} \underset{\Zb} { \times} \Yb_{rd}
%$$
%is relatively ergodic.
%Now a fortiori 
%the extension 
%$$
% \Xb_\phi \underset{\Zb} { \times} \Xb_\phi \to \Yb \underset{\Zb} { \times} \Yb
%$$
%is relatively ergodic.
 \end{proof}

\br

\br

%---------------

%A special case of this is as follows:
%
%Taking $\Xb = \Yb$ in Conjecture \ref{general} we arrive at :
%
%
%\begin{conj}\label{general0}
%Let $\Yb \to \Zb$ be an infinite factor of the ergodic system $\Yb$ 
%and let $G$ be a compact second countable topological group.
%Then for a generic $\phi \in \Ccal(Y,G)$, with 
%$\Yb_\phi = (Y \times G, \Ycal \times \Bcal_G, \mu \times \la_G, T_\phi)$,
%the extension 
%$$
% \Yb_\phi \underset{\Zb} { \times} \Yb_\phi \to \Yb \underset{\Zb} { \times} \Yb
%$$
%is relatively ergodic.
%Equivalently, the system $\Yb_\phi$ is $2$-fold ergodic over $\Yb \to \Zb$.
%\end{conj}
%
%
%In fact, we can ``almost prove" this conjecture  as follows: 
%Let $\tilde{\Zb} \to \Zb$ be the maximal compact extension of $\Zb$ in $\Yb$.
%If this factor map is trivial, that is if the extension $\Yb \to \Zb$ is relatively weakly mixing,
%the conjecture holds by Theorem \ref{rwm}. In the case where $\tilde{\Zb} \to \Zb$
% is a $K$-extension, the conjecture follows from Theorem \ref{max2}.
%Thus the only remaining case is when $\tilde{\Zb} \to \Zb$ is a nontrivial 
%homogeneous extension (as in Theorem \ref{homog}). It seems that the conjecture 
%holds in this remaining case as well but we do not resolve this question here.
%
%
%\br
%------------------

We now observe that:

\begin{itemize}
\item
Theorem \ref{max2} is a special case of Theorem \ref{general}, when we assume that the 
extension $\Yb \to \Zb$ is a $K/H$-extension.
\item
Theorem \ref{max} is the special case of Theorem \ref{general} when we assume that  
$\Xb = \Yb$ and that the extension $\Yb \to \Zb$ is a $K/H$-extension.
\item
Theorem \ref{rwm} is obtained from Theorem \ref{general} when we take $\Zb = \Yb$.
\item
By taking $\Xb = \Yb$ in Theorem \ref{general} we arrive at Theorem \ref{generall}
whose conclusion can be stated as the claim that the system $\Yb_\phi$ is $2$-fold ergodic over $\Yb \to \Zb$.
\end{itemize}

%The next conjecture would follow as well from Conjecture \ref{general}:
%
%\begin{conj}\label{rwm}
%If $\Yb \to \Zb$ is a relatively weakly mixing extension, then
%for a generic cocycle $\phi \in \Ccal(Y,G)$, the corresponding skew product transformation $T_\phi$
%is such that the extension $\Yb_\phi \to \Zb$ is a relatively weakly mixing extension.
%\end{conj}
%
%\br
%
%In turn, conjecture \ref{rwm} implies the following :
%%a weak form of Theorem \ref{general} :

\br

\section{Appendix : Three lemmas}

\br

{\bf $\bullet$ A lemma about random permutations} 

\br

\begin{lem}
For any $0 < \ga <1$ 
$$
\Prob \{\pi \in \Sym(N) : | \{i \le \lfloor \ga N \rfloor : \pi(i) \le   \lfloor \ga N \rfloor \} | > \frac12 \ga^2 N \} 
> 1 - \frac{10(1-\ga)}{N \ga^2}.
$$
\end{lem}

\begin{proof}
Define random variables $Y_i$ for $1 \le i \le  M := \lfloor \ga N \rfloor $ by
%$$
%y_i(\pi) = \begin{cases}
%1 & {\text{if}} \  \pi(i) \leq  \lfloor \ga N \rfloor\\
%0  & {\text{if}} \  \pi(i) > \lfloor \ga N \rfloor. 
%\end{cases}
%$$
$$
Y_i(\pi) = \begin{cases}
1 & {\text{if}} \  \pi(i) \leq M\\
0  & {\text{if}} \  \pi(i) >  M. 
\end{cases}
$$
Clearly $\Prob (Y_i = 1) = \varrho := \frac{M}{N}$, and for $i \not = j$
$$
\Prob (Y_i Y_j = 1) = \frac{M(M-1)}{N^2}.
$$
We now have
\begin{align*}
\E \left(\left( \frac{1}{M} \sum_{i=1}^M ( Y_i - \varrho)\right)^2\right) & = 
\frac{1}{M^2} \sum_{i=1}^M \E \left(\left(Y_i - \varrho\right)^2\right)
+ \frac{1}{M^2} \sum_{i \not =j} \E \left(\left(Y_i - \varrho \right) \cdot \left(Y_j - \varrho \right)\right) \\
& =
\frac{\varrho(1 - \varrho)}{M} + \frac{1}{M^2} \sum_{i \not =j} \left( \frac{M(M -1)}{N^2} - \varrho^2 \right)\\
& = \frac{\varrho(1 - \varrho)}{M}  - \frac{1}{M^2} \cdot \frac{M}{N^2}  <  \frac{\varrho(1 - \varrho)}{M}.
\end{align*}
It follows by Chebytchef's iequality that
$$
\Prob \left( \left |\frac{1}{M} \sum_{i=1}^M (Y_i -\varrho) \right| 
> \frac{\varrho}{3} \right) \leq \frac{9}{\varrho^2} \cdot \frac{\varrho(\varrho-1)}{M}
$$
and thus
$$
\Prob \left(\frac{1}{M} \sum_{i=1}^M Y_i- \varrho \geq  -\frac{\varrho}{3} \right) 
\geq  1 - \frac{9}{\varrho^2} \cdot \frac{\varrho(1-\varrho)}{M}, 
$$
or
$$
\Prob \left( \sum_{i=1}^M Y_i  \geq  \frac{2}{3}  \varrho M \right) 
\geq  1 - \frac{9}{\varrho} \cdot \frac{1-\varrho}{M} \geq 1 - \frac{10(1 - \ga)}{\ga^2 N}.
$$
%Now
%$\frac{10(1 - \ga)}{\ga^2 N} \geq \frac{9(1 - \varrho)}{\varrho^2 N}$, hence 
%$\frac{10}{9} > \frac{\ga^2}{\varrho^2} \frac{1- \varrho}{1 - \ga}$,
%from which the lemma follows for all sufficiently large $N$.
\end{proof}

\br

{\bf $\bullet$ The $\sqrt{\del}$ to $\del$ lemma} 

\br

\begin{lem}
Let $\Pcal = \{P_i : i \in I\}$ be a finite partition of a probability space $(\Om,\mu)$.
Let $E \subset \Om$ with $\mu(E)  < \del$,
then for
$$
I_0 = \{i : \frac{\mu(P_i \cap E)}{\mu(P_i)} > \sqrt{\del}\},
$$
we have then $\sum_{i \not \in I_0} \mu(P_i) > 1 - \sqrt{\del}$.
\end{lem}

\begin{proof}
For $i \in I_0$, we have $\mu(P_i \cap E) > \sqrt{\del} \mu(P_i)$, whence
$$
\del > \sum_{i \in I_0} \mu(P_i \cap E) > \sqrt{\del} \sum_{i \in I_0} \mu(P_i),
$$
so that $\sqrt{\del} > \sum_{i \in I_0} \mu(P_i)$.
\end{proof}

 \br
 
 {\bf  $\bullet$ A simple probabilistic lemma} 

\br

\begin{lem}
Let $p$ be a number in $(0,1)$ and let $L$ be a positive constant.
Suppose that $w_j \geq 0$ and $\sum_{j =1}^n w_j =1$.
Let $X_j, \ j=1,\dots,n$ be random variables taking values in $\{0,1\}$ with $P(X_j =1) \geq p$.
Suppose further that for each $j$ there is a set $I_j$, \  $ |I_j|  \leq L$,
such that for $i \not\in I_j$, $X_i$ and $X_j$ are independent.
Set $w = \max_{1 \leq j \leq n} w_j$.
Then for $X = \sum_{j=1}^n w_j X_j$ we have
$$
P(X \geq p/2)  \geq 1 - \frac{4(L+1)}{p^2} w.
$$
\end{lem}

\begin{proof}
Let $\bar{x}_j = P(X-j =1)$ and define $Y_j = X_j - \bar{x}_j 
$.
Let $\bar{p} = \sum_{j=1}^n w_j \bar{x}_j \geq p$.
Then $\sum w_j Y_j = X - \bar{p}$, and
if $j \not \in I_j$ we have
$$
\E(Y_j Y_i) =0, \quad {\text{while}} \quad \E(X - p) =0.
$$
Now
$$
(X - \bar{p})^2 = \sum_{j=1}^n w_j^2 Y_j^2 + \sum_{j=1}^n \sum_{i \in I_j} Y_j Y_i.
$$
Since $\E(Y_j^2) \leq 1/4$ we get
$$
\E( (X - \bar{p})^2) \leq 1/4 \cdot w + w \cdot L  \leq w(L+1).
$$
As $\bar{p} \geq p$, our claim follows.
\end{proof}

\end{document}